\documentclass[11pt]{article}

\usepackage{amsmath}
\usepackage{amssymb}
\usepackage{amsfonts}
\usepackage{mathrsfs}
\usepackage{graphicx}
\usepackage{color,xcolor}
\usepackage{float}
\usepackage{subfig}
\usepackage[normalem]{ulem}
\usepackage[linktocpage,colorlinks,linkcolor=blue,anchorcolor=blue, citecolor=blue,urlcolor=blue]{hyperref}
\usepackage{pdfpages}
\usepackage[inline]{enumitem}
\usepackage{multirow}
\usepackage{makecell}
\usepackage{breakcites}
\usepackage{bbm}
\usepackage{wrapfig}
\usepackage{wrapstuff}
\usepackage[thinc]{esdiff}
\allowdisplaybreaks[4]
\usepackage{rotating}
\usepackage{scalerel,stackengine}

\newtheorem{theorem}{Theorem}[section]
\newtheorem{definition}[theorem]{Definition}
\newtheorem{lemma}[theorem]{Lemma}
\newtheorem{example}[theorem]{Example}
\newtheorem{corollary}[theorem]{Corollary}

\newtheorem{proposition}[theorem]{Proposition}
\newtheorem{fact}[theorem]{Fact}

\newtheorem{remark}[theorem]{Remark}

\def\bd{\boldsymbol}
\def\pn{\par\smallskip\noindent}

\DeclareMathAlphabet\mathbfcal{OMS}{cmsy}{b}{n}

\newcommand{\N}{\mathbb{N}}
\newcommand{\R}{\mathbb{R}}

\newcommand{\F}{\mathcal{F}}

\newcommand{\bbB}{\mathbb{B}}
\newcommand{\bbO}{\mathbb{O}}

\newcommand{\G}{\mathcal{G}}

\newcommand{\HI}{\mathcal{H}}

\newcommand{\ba}{\boldsymbol{a}}
\newcommand{\bA}{\boldsymbol{A}}

\newcommand{\bb}{\boldsymbol{b}}
\newcommand{\bc}{\boldsymbol{c}}

\newcommand{\bB}{\boldsymbol{B}}
\newcommand{\bC}{\boldsymbol{C}}

\newcommand{\bI}{\boldsymbol{I}}
\newcommand{\bJ}{\boldsymbol{J}}
\newcommand{\be}{\boldsymbol{e}}

\newcommand{\bh}{\boldsymbol{h}}
\newcommand{\bH}{\boldsymbol{H}}
\newcommand{\bO}{\boldsymbol{O}}
\newcommand{\bx}{\boldsymbol{x}}
\newcommand{\by}{\boldsymbol{y}}
\newcommand{\bz}{\boldsymbol{z}}

\newcommand{\bv}{\boldsymbol{v}}
\newcommand{\bw}{\boldsymbol{w}}

\newcommand{\bp}{\boldsymbol{p}}
\newcommand{\bP}{\boldsymbol{P}}
\newcommand{\bq}{\boldsymbol{q}}
\newcommand{\bQ}{\boldsymbol{Q}}
\newcommand{\bs}{\boldsymbol{s}}
\newcommand{\bS}{\boldsymbol{S}}

\newcommand{\bW}{\boldsymbol{W}}
\newcommand{\bX}{\boldsymbol{X}}
\newcommand{\bY}{\boldsymbol{Y}}
\newcommand{\bZ}{\boldsymbol{Z}}
\newcommand{\bbU}{\mathbb{U}}
\newcommand{\bbV}{\mathbb{V}}

\newcommand{\TP}{\mathcal{P}}

\newcommand{\T}{\textnormal{T}}
\newcommand{\TF}{\textnormal{F}}

\newcommand\pig[1]{\scalerel*[5pt]{\big#1}{%
  \ensurestackMath{\addstackgap[1.5pt]{\big#1}}}}

\usepackage{booktabs}
\usepackage{graphicx}
\usepackage{stmaryrd}
\usepackage{mathtools}
\def\multiset#1#2{\ensuremath{\left(\kern-.3em\left(\genfrac{}{}{0pt}{}{#1}{#2}\right)\kern-.3em\right)}}

\usepackage{multicol}
\usepackage{nicematrix}
\usepackage{stackengine}

\stackMath

\usepackage{algorithm}
\usepackage{algorithmic}

\usepackage{stackengine}
\stackMath
\newcommand\tsup[2][2]{%
 \def\useanchorwidth{T}%
  \ifnum#1>1%
    \stackon[-.5pt]{\tsup[\numexpr#1-1\relax]{#2}}{\scriptscriptstyle\sim}%
  \else%
    \stackon[.5pt]{#2}{\scriptscriptstyle\sim}%
  \fi%
}

\newcommand\restr[2]{{% we make the whole thing an ordinary symbol
  \left.\kern-\nulldelimiterspace % automatically resize the bar with \right
  #1 % the function
  \littletaller % pretend it's a little taller at normal size
  \right|_{#2} % this is the delimiter
  }}

\newcommand{\littletaller}{\mathchoice{\vphantom{\big|}}{}{}{}}

\usepackage{tikzit}
\tikzstyle{red dot}=[fill=red, draw=black, shape=circle]
\tikzstyle{green dot}=[fill=green, draw=black, shape=circle]
\tikzstyle{medium box}=[fill=white, draw=black, shape=rectangle, minimum width=0.75cm, minimum height=1cm]
\tikzstyle{0.1pt}=[fill=black, draw=black, shape=circle, minimum size=0.1pt, inner sep=0pt]
\tikzstyle{0.5pt}=[fill=black, draw=black, shape=circle, minimum size=0.5pt, inner sep=0pt]
\tikzstyle{1pt}=[fill=black, draw=black, shape=circle, minimum size=1pt, inner sep=0pt]
\tikzstyle{2pt}=[fill=black, draw=black, shape=circle, minimum size=2pt, inner sep=0pt]
\tikzstyle{3pt}=[fill=black, draw=black, shape=circle, minimum size=3pt, inner sep=0pt]
\tikzstyle{3pthide}=[fill=black, draw=black, shape=circle, minimum size=3pt, inner sep=0pt, opacity=0.1]
\tikzstyle{5pt}=[fill=black, draw=black, shape=circle, minimum size=5pt, inner sep=0pt]
\tikzstyle{big circle}=[fill=none, draw=black, shape=circle, minimum size=20cm, inner sep=0pt, ultra thick]
\tikzstyle{0.5none}=[fill=none, draw=none, shape=circle, scale=0.5]
\tikzstyle{small circle}=[fill=none, draw=black, shape=circle, minimum size=5cm, inner sep=0pt, ultra thick]

\tikzstyle{directional}=[>=stealth, ->]
\tikzstyle{thick direc}=[>=stealth, ->, very thick]
\tikzstyle{dashes}=[-, densely dotted, thick]
\tikzstyle{wavy}=[-, snake it, thick]
\tikzstyle{big dashes}=[-, thick, dashed, dash pattern=on 4mm off 2mm, fill=cyan]
\tikzstyle{blue directional}=[draw=blue, ->, very thick, >=stealth]
\tikzstyle{red directional}=[draw=red, ->, very thick, >=stealth]
\tikzstyle{green directional}=[draw=green, ->, very thick, >=stealth]
\tikzstyle{thin line}=[-, very thin]
\tikzstyle{purple directional}=[->, draw=magenta, very thick, >=stealth]
\tikzstyle{pure shade}=[-, draw=none, fill={rgb,255: red,191; green,191; blue,191}, opacity=0.5]
\tikzstyle{pure rshade}=[-, draw=none, fill={rgb,255: red,255; green,0; blue,0}, opacity=0.5]
\tikzstyle{pure gshade}=[-, draw=none, fill={rgb,255: red,0; green,255; blue,0}, opacity=0.5]
\tikzstyle{blue line}=[-, fill=none, draw=blue, very thick]
\tikzstyle{purple line}=[-, draw=magenta, very thick]
\tikzstyle{line shade}=[-, draw=black, fill={rgb,255: red,191; green,191; blue,191}]
\tikzstyle{line shade 2}=[-, draw=black, fill={rgb,255: red,128; green,128; blue,128}]
\tikzstyle{thick line}=[-, very thick]
\tikzstyle{tlhide}=[-, very thick, opacity=0.1]

\usepackage{tikz}
\usepackage{tikz-3dplot}
\usetikzlibrary{calc}
\usepackage{pgfplots}
\usepackage{xxcolor}
\pgfplotsset{compat=1.16}
\usepgfplotslibrary{fillbetween}
\pgfdeclareradialshading[tikz@ball]{ball}{\pgfqpoint{0bp}{0bp}}{%
 color(0bp)=(tikz@ball!0!white);
 color(7bp)=(tikz@ball!0!white);
 color(15bp)=(tikz@ball!70!black);
 color(20bp)=(black!70);
 color(30bp)=(black!70)}
\makeatother

\tikzset{viewport/.style 2 args={
    x={({cos(-#1)*1cm},{sin(-#1)*sin(#2)*1cm})},
    y={({-sin(-#1)*1cm},{cos(-#1)*sin(#2)*1cm})},
    z={(0,{cos(#2)*1cm})}
}}

\pgfplotsset{only foreground/.style={
    restrict expr to domain={rawx*\CameraX + rawy*\CameraY + rawz*\CameraZ}{-0.05:100},
}}
\pgfplotsset{only background/.style={
    restrict expr to domain={rawx*\CameraX + rawy*\CameraY + rawz*\CameraZ}{-100:0.05}
}}

\def\addFGBGplot[#1]#2;{
    \addplot3[#1,only background, opacity=0.25] #2;
    \addplot3[#1,only foreground] #2;
}

\usetikzlibrary{3d}
\usetikzlibrary{calc}
\usetikzlibrary{babel} % for issues with some babel packages

\pgfmathsetmacro\xx{1/sqrt(2)}
\pgfmathsetmacro\xy{1/sqrt(6)}
\pgfmathsetmacro\zy{sqrt(2/3)}

\usetikzlibrary{3d}
\usetikzlibrary{calc, shadings} 
\usetikzlibrary{positioning,arrows.meta}
\usetikzlibrary{trees}
\usepgfplotslibrary{fillbetween}
\DeclareMathAlphabet\mathbfcal{OMS}{cmsy}{b}{n}

\makeatletter
\def\tikz@lib@cuboid@get#1{\pgfkeysvalueof{/tikz/cuboid/#1}}

\def\tikz@lib@cuboid@setup{%
   \pgfmathsetlengthmacro{\vxx}%
      {\tikz@lib@cuboid@get{xscale}*cos(\tikz@lib@cuboid@get{xangle})*1cm}
   \pgfmathsetlengthmacro{\vxy}%
      {\tikz@lib@cuboid@get{xscale}*sin(\tikz@lib@cuboid@get{xangle})*1cm}
   \pgfmathsetlengthmacro{\vyx}%
      {\tikz@lib@cuboid@get{yscale}*cos(\tikz@lib@cuboid@get{yangle})*1cm}
   \pgfmathsetlengthmacro{\vyy}%
      {\tikz@lib@cuboid@get{yscale}*sin(\tikz@lib@cuboid@get{yangle})*1cm}
   \pgfmathsetlengthmacro{\vzx}%
      {\tikz@lib@cuboid@get{zscale}*cos(\tikz@lib@cuboid@get{zangle})*1cm}
   \pgfmathsetlengthmacro{\vzy}%
      {\tikz@lib@cuboid@get{zscale}*sin(\tikz@lib@cuboid@get{zangle})*1cm}
}

\def\tikz@lib@cuboid@draw#1--#2--#3\pgf@stop{%
    \begin{scope}[join=bevel,x={(\vxx,\vxy)},y={(\vyx,\vyy)},z={(\vzx,\vzy)}]
       \begin{scope}[canvas is yz plane at x=#1]
          \draw[cuboid/all faces,cuboid/edges,cuboid/right face] 
                (0,0) -- ++(#2,0) -- ++(0,-#3) -- ++(-#2,0) -- cycle;
          \draw[cuboid/all grids,cuboid/right grid] (0,0) grid (#2,-#3);
       \end{scope}
       \begin{scope}[canvas is xy plane at z=0]
          \draw[cuboid/all faces,cuboid/edges,cuboid/front face] 
                (0,0) -- ++(#1,0) --  ++(0,#2) -- ++(-#1,0) -- cycle;
          \draw[cuboid/all grids,cuboid/front grid] (0,0) grid (#1,#2);
       \end{scope}
       \begin{scope}[canvas is xz plane at y=#2]
          \draw[cuboid/all faces,cuboid/edges,cuboid/top face] 
                (0,0) -- ++(#1,0) --  ++(0,-#3) -- ++(-#1,0) -- cycle;
          \draw[cuboid/all grids,cuboid/top grid] (0,0) grid (#1,-#3);
       \end{scope}
       \draw[cuboid/hidden edges] (0,#2,-#3) -- (0,0,-#3) -- (0,0,0) 
                (0,0,-#3) -- ++(#1,0,0);
       \begin{scope}[canvas is yz plane at x=#1]
          \draw[cuboid/all faces,cuboid/right face,cuboid/edges,fill opacity=0] 
                (0,0) -- ++(#2,0) -- ++(0,-#3) -- ++(-#2,0) -- cycle;
       \end{scope}
       \begin{scope}[canvas is xy plane at z=0]
          \draw[cuboid/all faces,cuboid/front face,cuboid/edges,fill opacity=0] 
                (0,0) -- ++(#1,0) --  ++(0,#2) -- ++(-#1,0) -- cycle;
       \end{scope}
       \begin{scope}[canvas is xz plane at y=#2]
          \draw[cuboid/all faces,cuboid/top face,cuboid/edges,fill opacity=0] 
                (0,0) -- ++(#1,0) --  ++(0,-#3) -- ++(-#1,0) -- cycle;
       \end{scope}
       \path (0,#2,0) coordinate (-left top front)
                      coordinate (-left front top)
                      coordinate (-top left front)
                      coordinate (-top front left)
                      coordinate (-front top left)
                      coordinate (-front left top);
       \path (0,#2,-#3) coordinate (-left top rear)
                        coordinate (-left rear top)
                        coordinate (-top left rear)
                        coordinate (-top rear left)
                        coordinate (-rear top left)
                        coordinate (-rear left top);
       \path (0,0,-#3) coordinate (-left bottom rear)
                       coordinate (-left rear bottom)
                       coordinate (-bottom left rear)
                       coordinate (-bottom rear left)
                       coordinate (-rear bottom left)
                       coordinate (-rear left bottom);
       \path (0,0,0) coordinate (-left bottom front)
                     coordinate (-left front bottom)
                     coordinate (-bottom left front)
                     coordinate (-bottom front left)
                     coordinate (-front bottom left)
                     coordinate (-front left bottom);
       \path (#1,#2,0) coordinate (-right top front)
                       coordinate (-right front top)
                       coordinate (-top right front)
                       coordinate (-top front right)
                       coordinate (-front top right)
                       coordinate (-front right top);
       \path (#1,#2,-#3) coordinate (-right top rear)
                         coordinate (-right rear top)
                         coordinate (-top right rear)
                         coordinate (-top rear right)
                         coordinate (-rear top right)
                         coordinate (-rear right top);
       \path (#1,0,-#3) coordinate (-right bottom rear)
                        coordinate (-right rear bottom)
                        coordinate (-bottom right rear)
                        coordinate (-bottom rear right)
                        coordinate (-rear bottom right)
                        coordinate (-rear right bottom);
       \path (#1,0,0) coordinate (-right bottom front)
                      coordinate (-right front bottom)
                      coordinate (-bottom right front)
                      coordinate (-bottom front right)
                      coordinate (-front bottom right)
                      coordinate (-front right bottom);
       \coordinate (-left center) at (0,.5*#2,-.5*#3);
       \coordinate (-right center) at (#1,.5*#2,-.5*#3);
       \coordinate (-top center) at (.5*#1,#2,-.5*#3);
       \coordinate (-bottom center) at (.5*#1,0,-.5*#3);
       \coordinate (-front center) at (.5*#1,.5*#2,0);
       \coordinate (-rear center) at (.5*#1,.5*#2,-#3);
       \coordinate (-center) at (.5*#1,.5*#2,-.5*#3);
       \path (0,#2,-.5*#3) coordinate (-left top center) 
                           coordinate (-top left center);
       \path (.5*#1,#2,-#3) coordinate (-top rear center)
                            coordinate (-rear top center);
       \path (#1,#2,-.5*#3) coordinate (-right top center)
                            coordinate (-top right center);
       \path (.5*#1,#2,0) coordinate (-top front center)
                          coordinate (-front top center);
       \path (0,0,-.5*#3) coordinate (-left bottom center) 
                           coordinate (-bottom left center);
       \path (.5*#1,0,-#3) coordinate (-bottom rear center)
                            coordinate (-rear bottom center);
       \path (#1,0,-.5*#3) coordinate (-right bottom center)
                            coordinate (-bottom right center);
       \path (.5*#1,0,0) coordinate (-bottom front center)
                          coordinate (-front bottom center);
       \path (0,.5*#2,0) coordinate (-left front center) 
                           coordinate (-front left center);
       \path (0,.5*#2,-#3) coordinate (-left rear center)
                            coordinate (-rear left center);
       \path (#1,.5*#2,0) coordinate (-right front center)
                            coordinate (-front right center);
       \path (#1,.5*#2,-#3) coordinate (-right rear center)
                          coordinate (-rear right center);
    \end{scope}
}

\tikzset{
  pics/cuboid/.style = {
    setup code = \tikz@lib@cuboid@setup,
    background code = \tikz@lib@cuboid@draw#1\pgf@stop
  },
  pics/cuboid/.default={1--1--1},
  cuboid/.is family,
  cuboid,
  all faces/.style={fill=white},
  all grids/.style={draw=none},
  front face/.style={},
  front grid/.style={},
  right face/.style={},
  right grid/.style={},
  top face/.style={},
  top grid/.style={},
  edges/.style={},
  hidden edges/.style={draw=none},
  xangle/.initial=0,
  yangle/.initial=90,
  zangle/.initial=210,
  xscale/.initial=1,
  yscale/.initial=1,
  zscale/.initial=0.5
}

\newcommand{\tikzcuboidreset}{
\tikzset{cuboid,
  all faces/.style={fill=white},
  all grids/.style={draw=none},
  front face/.style={},
  front grid/.style={},
  right face/.style={},
  right grid/.style={},
  top face/.style={},
  top grid/.style={},
  edges/.style={},
  hidden edges/.style={draw=none},
  xangle=0,
  yangle=90,
  zangle=210,
  xscale=1,
  yscale=1,
  zscale=0.5
}
}

\newcommand{\tikzcuboidset}{\@ifstar\tikzcuboidset@star\tikzcuboidset@nostar} 
\newcommand{\tikzcuboidset@nostar}[1]{\tikzcuboidreset\tikzset{cuboid,#1}}
\newcommand{\tikzcuboidset@star}[1]{\tikzset{cuboid,#1}}
\makeatother

\usetikzlibrary{angles, quotes}

\makeatletter
\newif\ifnobrackets
\renewcommand\@cite[2]{\ifnobrackets\else[\fi{#1\if@tempswa , #2\fi}\ifnobrackets\else]\fi\nobracketsfalse}

\makeatother

\usepackage{lmodern}
\usetikzlibrary{decorations.pathmorphing}
\tikzset{snake it/.style={decorate, decoration=snake}}

\begin{document}

\title{On subdifferential chain rule of matrix factorization and beyond}
% \title{Local surjection property and Clarke subdifferential chain rule of matrix factorization, factorization machine and beyond}
% On chain rule of overparameterized matrix factorization and beyond
% : A perspective from local surjection property

\author{
Jiewen GUAN
\thanks{Department of Systems Engineering and Engineering Management, The Chinese University of Hong Kong, Shatin, New Territories, Hong Kong. Email: seemjwguan@gmail.com}
    \and
Anthony Man-Cho SO
\thanks{Department of Systems Engineering and Engineering Management, The Chinese University of Hong Kong, Shatin, New Territories, Hong Kong. Email: manchoso@se.cuhk.edu.hk}
}

\date{\today}

\maketitle

\begin{abstract}
In this paper, we study equality-type Clarke subdifferential chain rules of matrix factorization and factorization machine.
% through a sufficient condition---the local surjection property.
% through the local surjection property.
% , a rarely-mentioned sufficient condition for the former.
Specifically, we show for these problems that provided
% (which can be quite arbitrary) 
the latent dimension is larger than some multiple of the problem size
(i.e., slightly overparameterized)
% the aforementioned desired property holds
% the mappings are locally surjective everywhere, 
and the loss function is locally Lipschitz, the subdifferential chain rules hold everywhere. In addition, we examine the tightness of the analysis
% conditions for the underpinning local surjection property
through some interesting constructions
% we also accompany the above results with some negative examples showing the tightness of the analysis
% and 
and make some important observations from the perspective of optimization; e.g., we show that for all this type of problems, computing a stationary point is trivial. Some tensor generalizations and neural extensions are also discussed, albeit they remain mostly open.
% On top of these, many attempts have also been made to generalize them to the case of tensor CP factorization, despite only some partial results are obtained, and as applications, some neural extensions are also established or outlooked, revealing their close connections to our studies.

\vspace{0.25cm}

\noindent {\bf Keywords:} nonconvex optimization, nonsmooth analysis, subdifferential chain rule, local surjection property, openness, matrix factorization, factorization machine, tensors, neural networks

\vspace{0.25cm}

\noindent {\bf Mathematics Subject Classification (2020):} 
15A23, % Factorization of matrices
% 15A60, % Norms of matrices, numerical range, applications of functional analysis to matrix theory
47A07, % Forms (bilinear, sesquilinear, multilinear)
% 47A12, % Numerical range, numerical radius
49J52, % Nonsmooth analysis 
49J53, % Set-valued and variational analysis
90C26 % Nonconvex programming, global optimization
\end{abstract}

    \section{Introduction}\label{sec:introduction}
    % In modern deep learning regimes, many interesting and impressive analyzes essentially rely on the validity of equality-type subdifferential chain rules.
    Strictly speaking, many existing ``subgradient'' algorithms for training deep neural networks (e.g.,~\cite{soltanolkotabi2017learning,arora2019fine}) are not necessarily subgradient-based in any conventional sense due to the possible failure of equality-type subdifferential chain rules for the underlying problems.
    % This is mainly because the equality-type subdifferential chain rules, which are casually abused in these works to define subgradients, are not always established in reality, due to the absence of subdifferential regularity~\cite[Definition~7.25]{rockafellar2009variational}:
    % In general, one can only use the existing results in~\cite[Section~2.3]{clarke1990optimization} and~\cite[Section~10.B]{rockafellar2009variational} to obtain a subset that can oftentimes be empty for the Fr\'echet subdifferential, as well as some supersets that might be much larger than expected for the limiting and Clarke ones. 
    Consequently, it is often difficult to establish rigorous theoretical guarantees for these algorithms.
    % Despite 
    % % (at least for two-layer ReLU neural networks)
    % % it is known that
    % % the general position condition~\cite[Definition~54]{tian2023testing}, which is sufficient to imply the 
    % % the subdifferential chain rules
    % they 
    % actually generically hold for such problems
    % % (see, e.g.,~\cite[Section~4.3]{tian2023testing}), 
    % due to the celebrated Rademacher’s Theorem~\cite[Theorem~9.60]{rockafellar2009variational},
    % % ~\cite[Proposition~19]{tian2023testing}, holds with probability one provided the input data are generated from some absolutely continuous probability measure (w.r.t.\ the Lebesgue one)~\cite[Section~4.3]{tian2023testing}, as assumed in all aforementioned works~\cite{xie2017diverse,soltanolkotabi2017learning,zhong2017recovery,arora2019fine}, and despite these methods also work pretty well in practice, 
    % there still lacks some mathematical rigorousness
    % % , transparency
    % and aesthetics in our current understanding, hindering us from being clearer and sharper. 
    Therefore, it is meaningful to 
    study conditions under which equality-type subdifferential chain rules hold for nonsmooth functions that arise in applications.
    % develop theories with this regard. 
    Apart from this perspective, subdifferential chain rules are also helpful in some other tasks, e.g., stationarity testing~\cite{yun2018efficiently} and landscape analysis~\cite{guan2024subdiff}. Unfortunately, the research along this direction is still quite limited.
    % , and therefore, developing theories with this regard is a must.
    
    % 123\footnote{Although this is very intuitive, it is actually not completely trivial. To see this, we first consider the case where there are $d+1$ points in $\R^d$ for some $d\in\N$. Suppose they happen to fall on some hyperplane, then the simplex generated by them is also included within the same hyperplane, and thus has zero Lebesgue measure. By the classical formula for evaluating the Lebesgue measure of simplices under this particular setting~\cite{stein1966note}, we know the determinant of the matrix whose columns are composed by the $d+1$ augmented vector representations of the points where each one is appended with a one after the last coordinate must be zero. This, combined with the fact that the zero set of any polynomial function that is not identically zero has zero Lebesgue measure~\cite[Lemma~1]{okamoto1973distinctness}, implies the desired claim. For handling more points, a simple union bound argument would suffice.}
    
    % In 2023, a groundbreaking work~\cite[Section~4]{tian2023testing} discovers
    In the recent work~\cite{tian2023testing},
    an efficiently-verifiable necessary and sufficient data qualification
    % intuitive, neat and condition that is both 
    for the validity of subdifferential chain rules
    % (in terms of either Fr\'echet, limiting and Clarke subdifferential)
    in training two-layer ReLU neural networks is developed.
    % ; see Section~4 therein for more details. 
    % This work not only settles a long-standing important problem in the field, but also raises 
    Motivated by this, it is very natural to ask if we can further extend these theories to other more sophisticated neural networks, such as those carrying some multilinear structures (instead of the linearity underlying the ReLU neural networks),
    % . We remark that such neural networks
    which are prevalent in recommender systems studies; see, e.g.,~\cite{socher2013reasoning,blondel2016higher,he2017neural,he2017neuralfm,xiao2017attentional,he2018nais,xin2019relational,fan2019graph,chen2021neural}. 
    However, due to the multilinear structure involved, the situation becomes drastically different from the linear setting; e.g., the Jacobian of any multilinear mapping at the origin is always zero, which prevents the application of the well-known surjectivity condition for subdifferential chain rules (see, e.g.,~\cite[Exercise~10.7]{rockafellar2009variational}). As a consequence, to achieve the goal, new techniques have to be developed.
    % Towards tackling this problem, as a first step, it is very tempting to apply some existing sufficient conditions such as the ones and~\cite[Theorem~2.3.10]{clarke1990optimization} pertaining to the surjectivity of the Jacobian of the model to gain some partial understanding on under what circumstances do equality-type subdifferential chain rules hold. However, a quick attempt shows that even this simple idea (i.e., showing surjectivity) is completely infeasible, since for any multilinear mapping, its Jacobian at the origin is always zero, which does not convey any useful information for our purposes. This suggests the hardness of the problem even at the initial stage. 
    % Fortunately, for the Clarke subdifferential, a very classical result by Clarke himself~\cite[Theorem~2.3.10]{clarke1990optimization} reveals that another weaker but rather unpopular condition, the so-called local surjection property~\cite{ioffe1987local}, is also sufficient for the equality-type Clarke subdifferential chain rule to hold, pointing out a new direction for us to explore.

    % Motivated by the above observation
    Towards a thorough understanding of this generalization, as a first step, in this paper, we study equality-type Clarke subdifferential chain rules of matrix factorization (MF)~\cite{haeffele2019structured} and factorization machine (FM)~\cite{rendle2010factorization}, which are prototypes of the aforementioned neural networks.
    % , and tensor CP factorization~\cite{kolda2009tensor}, with an ultimate aim at establishing some sufficient conditions on the equality-type Clarke subdifferential chain rule for the aforementioned multilinear-structured neural networks~\cite{socher2013reasoning,blondel2016higher,he2017neural,he2017neuralfm,xiao2017attentional,he2018nais,xin2019relational,fan2019graph,chen2021neural} (with two layers).
    % The reason for studying these prototypical examples is rather simple: 
    % All these networks are somewhat generalized from them (or something more directly, the MF itself is already a two-layer linear neural network), and as we shall see later, some desired neural results can almost directly be implied by their corresponding prototypical counterparts. 
    Specifically, we show for these problems that provided the latent dimension is larger than some multiple of the problem size (i.e., slightly overparameterized) and the loss function is locally Lipschitz, the subdifferential chain rules hold everywhere.
    % we show for these problems that, provided the dimension of their latent representations is larger than some multiple of the problem size, their local surjection property (and thus their equality-type Clarke subdifferential chain rules under arbitrary but locally Lipschitz loss functions) holds everywhere. 
    Underpinning the theories is an interesting property called the local surjection property (a.k.a.\ openness), which suffices to guarantee the validity of equality-type Clarke subdifferential chain rules and can hold even in absence of surjectivity.
    % \footnote{We remark that for linear mappings, surjectivity suffices to guarantee the local surjection property, due to the open mapping theorem.}.
    In addition, we examine the tightness of the analysis through some interesting constructions and make some important observations from the perspective of optimization; e.g., we show that for all this type of problems, computing a stationary point is trivial, 
    % which challenges the reasonableness
    which 
    sheds new light on the usefulness
    % drives us to rethink about the meaningfulness
    of the convergence guarantees of many existing algorithms, such as~\cite[Theorem~5(a)]{hastie2015matrix},~\cite[Theorem~4.2]{li2015convergence},~\cite[Theorem~1]{wang2017polynomial},~\cite[Theorem~2]{lin2017robust}, and~\cite[Corollaries~1-2]{zhu2018dropping}.
    % ~\cite[Proposition~1]{vandaele2016efficient},
    % In addition, some negative examples showing the tightness of the above results are also carefully constructed, and some interesting remarks of the above results from the perspective of optimization are also discussed (Corollary~\ref{cor:zero-subdiff} and Remark~\ref{rmk:caveat}). 
    To go one step further, we attempt to generalize the developments to the tensor CANDECOMP/PARAFAC (CP) factorization. Despite only some partial results are obtained,
    % (Propositions~\ref{prop:surj-CP-origin},~\ref{prop:pseudo-tensors}), 
    % due to some tricky obstacles caused by the coordinate dependence of the multi-vector inner product used therein that we are at the moment unable to address, 
    they open up several interesting and important
    % but challenging
    directions for future studies. Besides, some preliminary applications to the aforementioned neural networks are also discussed.

    The rest of the paper is organized as follows. We first introduce the notation and preliminaries in Section~\ref{sec:notation-prelim}. Then, in Section~\ref{sec:main-results}, we first study the local surjection property and Clarke subdifferential chain rules of the prototypical examples, and then discuss their tensor generalizations and neural extensions. Finally, we conclude this paper with several future directions in Section~\ref{sec:conclusion}.

    \section{Notation and preliminaries}\label{sec:notation-prelim}

    \subsection{Notation}\label{sec:notations}
    Throughout this paper, unless stated otherwise, we adopt lowercase letters (e.g., $x$), boldface lowercase letters (e.g., $\bx=(x_i)$), boldface capital letters (e.g., $\bX=(x_{i j})$), and boldface calligraphic letters (e.g., $\mathbfcal{X}=(x_{i_1 i_2 i_3})$) to denote scalars, vectors, matrices, and third-order tensors,\footnote{In this paper we only study third-order tensors as the higher-order cases can be generalized in a straightforward manner.
    % This is because the corresponding results to be developed can all be straightforwardly generalized to higher-order cases.
    } 
    respectively. 
    % In rare cases where the meaning of the subscript can be ambiguous, we also use notations like $x_{i,j,k}$ to denote the elements. 
    % Lowercase Greek letters are usually denoted for constants and $\delta$'s, specifically, for universal constants. 
    % Denote $\R^{n_1\times n_2\times\dots\times n_d}$ to be the space of real tensors of order $d$ with dimension $n_1\times n_2\times\dots\times n_d$. The same notation applies to a vector space and a matrix space when $d=1$ and $d=2$, respectively. 
    Denote $\N$ to be the set of positive integers, and we fix some $m,n\in\N$ in subsequent introductions. We define $[n]:=[1,n]\cap\N$,
    % and $\BQ$ to be the set of rationals. 
    % In particular, all blackboard bold capital letters denote sets, such as $\R^n$, the standard basis $\BE^n:=\{\be_1,\be_2,\dots,\be_n\}$ of $\R^n$, the $\ell_p$-sphere $\BS^n_p:=\{\bx\in\R^n:\|\bx\|_p=1\}$. The superscript $n$ of a set, specifically, indicates that the concerned set is a subset of $\R^n$. 
and use $\|\cdot\|$ to denote the (vector) Euclidean norm, while $\|\cdot\|_{\TF}$ the (matrix/tensor) Frobenius norm.
% We denote $D(\bx)\in\R^{n\times n}$ to be the matrix whose diagonal vector is $\bx$ and off-diagonal entries are zeros. For a square matrix $X\in\R^{n\times n}$, we denote $D(X)\in\R^{n\times n}$ to be the matrix by keeping only the diagonal vector of $X$, i.e., replacing off-diagonal entries with zeros. 
The $n$-dimensional all-zero vector is denoted by $\bd{0}_n$ and all-one vector by $\bd{1}_n$. Besides, we use $\be^n_i$ to denote the $i$-th standard basis vector in $\R^n$. The $n\times n$ identity matrix is denoted by $\bI_n$ and $m\times n$ zero matrix by $\bO_{m,n}$. In $\R^n$, we use $\bbB^n:=\{\bx\in\R^n:\|\bx\|\le 1\}$
% and $\mathbb{S}^n:=\{\bx\in\R^n:\|\bx\|= 1\}$
to denote the unit ball, and 
we denote the unit ball in
% the same notation applies to 
matrix and tensor spaces similarly.
% and sphere, respectively. 
The subscripts and superscripts of these special vectors, matrices, and sets are often omitted as long as there is no ambiguity. 
% For symmetric matrices $A,B\in\R^{n\times n}$, $A\succeq O$ indicates that $A$ is positive semidefinite and $A\succeq B$ indicates that $A-B\succeq O$. 
% The Frobenius inner product between two tensors $\U,\V\in\R^{n_1\times n_2\times\dots\times n_d}$ is defined as
% \[
% \langle\U, \V\rangle := \sum_{i_1=1}^{n_1}\sum_{i_2=1}^{n_2} \dots\sum_{i_d=1}^{n_d} u_{i_1i_2\dots i_d} v_{i_1i_2\dots i_d}.
% \]
% Its induced Frobenius norm is naturally defined as $\|\TT\|:=\sqrt{\langle\TT,\TT\rangle}$. The two terms automatically apply to tensors of order two (matrices) and tensors of order one (vectors) as well. This is the conventional norm (a norm without a subscript) used throughout the paper.
Three vector operations are used frequently, namely 
% the outer product $\otimes$, 
the Kronecker product $\boxtimes$, appending vectors $\vee$, and the Hadamard product $\oast$. In particular, if $\bx\in\R^{m}$ and $\by\in\R^{n}$, then
% $\bx\otimes\by=\bx\by^{\T}\in\R^{m\times n}$, 
$$
\bx\boxtimes\by=(x_1\by^{\T},x_2\by^{\T},\dots,x_{m}\by^{\T})^{\T}\in\R^{m n},\quad\bx\vee\by=(x_1,x_2,\dots,x_{m},y_1,y_2,\dots,y_{n})^{\T}\in\R^{m+n},
$$
and
$$
    \bx\oast\by=(x_1 y_1,x_2 y_2,\dots,x_n y_n)^{\T}\in\R^n,\quad\text{if in addition $m=n$}.
$$
Besides, we also use $\vee$ to append matrices and tensors vertically.
% These three operators also apply to vector sets via element-wise operations.
% The notion $\Omega(f(n))$ means the same order of magnitude to $f(n)$, i.e., there exist positive constants $\alpha,\beta$ and $n_0$ such that $\alpha f(n)\le\Omega(f(n))\le\beta f(n)$ for all $n\ge n_0$. As a convention, the notation $\mathcal{O}(f(n))$ means at most the same order of magnitude to $f(n)$.
% For convenience, we also 
% % borrow notations from~\cite[Section~3]{li2020norm} to 
% generalize $\vee$ to allow appending along multiple dimensions by designating multiple subscripts; as an example, we have that $\bX=\bigvee_{i,j=1}^{m, n}x_{i j}$ for any $\bX\in\R^{m\times n}$. To avoid ambiguity, we use (say) $\bigvee_{i=1}^m(\bigvee_{j=1}^n x_{i j})$ (or just $\bigvee_{i=1}^m \bigvee_{j=1}^n x_{i j}$) to denote the vector in $\R^{mn}$ appended by the $m$ subvectors $\bigvee_{j=1}^n x_{i j}\in\R^n$ in order. 
We use $\sum_{j>i}$ and $\sum_{i=1}^{n-1}\sum_{j=i+1}^{n}$ interchangeably whenever $n$ is clear from the contexts, and the same convention applies to other operations such as appending and product. 
For any $\bx,\by,\bz\in\R^n$, we use $\langle\bx,\by\rangle:=\sum_{i=1}^n x_i y_i$ to denote the inner product of $\bx$ and $\by$, with a higher-order override $\langle\bx,\by,\bz\rangle:=\sum_{i=1}^n x_i y_i z_i$, and define $\operatorname{supp}(\bx):=\{i\in[n]:x_i\neq 0\}$. Besides, we denote $\operatorname{diag}(\bx)\in\R^{n\times n}$ to be the matrix whose diagonal vector is $\bx$ and off-diagonal entries are zeros.
For any $\bX\in\R^{m\times n}$, we use $\bx_i$ to denote its $i$-th column, and $\operatorname{vec}(\bX):=\bigvee_{i=1}^n\bx_i$ its standard vectorization. 
% and $\operatorname{ker}(\bX):=\{\by\in\R^n:\bX \by=\bd{0}\}$ to denote its kernel (a.k.a., nullspace). 
% ; see~\cite{li2020norm} for a comprehensive introduction for this notation.
    For a mapping $\F:\R^n\rightarrow\R^m$, 
    % we use $\operatorname{Gph}(\F):=\{(\bx,\F(\bx))\in\R^{m+n}:\bx\in\R^n\}$ to denote its graph, and 
    we use $\bJ_{\F}(\bx)\in\R^{m\times n}$ to denote its Jacobian at $\bx\in\R^n$. 
    % (For a function $f:\R^n\rightarrow\R$, we use $\nabla f(\bx)\in\R^n$ to denote its gradient at $\bx\in\R^n$ instead; this should be distinguished from the Jacobian.) 
    For a nonzero number $x\in\R$, we use $\operatorname{sign}(x):=x/|x|$ to denote its sign; such an operation is applied in an elementwise manner when the argument is a vector. For any set $\mathbb{X}\subseteq\R^n$, we use $\operatorname{int}(\mathbb{X})$ and $\operatorname{relint}(\mathbb{X})$ respectively to denote its interior and relative interior, $\operatorname{span}(\mathbb{X})$ the subspace spanned by $\mathbb{X}$,
    % , i.e., the set of all linear combinations of the vectors in $\mathbb{X}$, 
    $\operatorname{conv}(\mathbb{X})$ its convex hull, $\operatorname{dim}(\mathbb{X})$ its dimension, and whenever $\mathbb{X}$ is a subspace, we use $\mathcal{P}_{\mathbb{X}}(\cdot)$ to denote its orthogonal projector and $\mathbb{X}^\perp$ its orthogonal complement. When $\mathbb{X}$ is a singleton set, we also identify it with its only element. We use $\circ$ to denote function composition and $\varnothing$ the empty set. For any $\mathbfcal{X}\in\R^{n_1\times n_2\times n_3}$, we call the matrix obtained by fixing its mode-$k$ index to $i$ the $i$-th mode-$k$ slice, and we use $\operatorname{squeeze}(\mathbfcal{X})$ to denote the resulting object by removing all its dimensions with length one, which has the same elements as $\mathbfcal{X}$.
    % For a function $f:\R^n\rightarrow\R$, we say it

    \subsection{Generalized differentiation theory}\label{sec:diff-theory}
    % In this paper, we will encounter three different notions of subdifferentials, namely Fr\'echet, limiting and Clarke, denoted by $\widehat{\partial} f(\bx)$, $\partial f(\bx)$ and $\partial_C f(\bx)$ for some function $f:\R^n\rightarrow\R$ at some point $\bx\in\R^n$, respectively. As their concrete definitions are not used throughout the paper, we will not elaborate them here, but only refer interested readers to
    In this subsection, we first introduce three different notions of subdifferentials for locally Lipschitz functions (which is the case for all examples to be studied) that will be used throughout the paper.
    % We shall commence at the one that is most similar to the convex subdifferential:

    \begin{definition}\label{def:subdiffs}
    % [Subdifferentials]
    Given a locally Lipschitz function $f:\R^n\rightarrow\R$ and a point $\bx\in\R^n$:
    \begin{itemize}[leftmargin=*]
        \item The Fr\'echet subdifferential~\cite[Definition~8.3(a)]{rockafellar2009variational} of $f$ at $\bx$ is defined as
        $$
            \widehat{\partial} f(\bx):=\left\{\bs \in \R^n: \liminf _{\substack{\by \rightarrow \bx,\, \by\neq\bx}} \frac{f(\by)-f(\bx)-\langle\bs, \by-\bx\rangle}{\|\by-\bx\|} \geq 0\right\}.
        $$
        \item The limiting subdifferential~\cite[Definition~8.3(b)]{rockafellar2009variational} of $f$ at $\bx$ is defined as
        $$
        \partial f(\bx):=\limsup_{\substack{\by\rightarrow\bx ,\, f(\by)\rightarrow f(\bx)}}\widehat{\partial} f(\by),
        $$
        % , i.e., the outer limit~\cite[Section~5.B]{rockafellar2009variational} of the Fr\'echet subdifferential, 
        where the $\limsup$ is taken in the sense of 
        % convergence of set-valued mappings
        Painlev{\'e}-Kuratowski~\cite[Section~5.B]{rockafellar2009variational}.
        \item The Clarke subdifferential~\cite[Theorem~2.5.1]{clarke1990optimization},~\cite[Theorem~9.61]{rockafellar2009variational} of $f$ at $\bx$ is defined as
        $$
        \partial_C f(\bx):=\operatorname{conv}\left(\limsup_{\by\rightarrow\bx,\,\by\in\mathbb{D}}\{\nabla f(\by)\}\right),
        % \by\xxrightarrow[\mathbb{D}]{}\bx
        $$
        where $\mathbb{D}\subseteq\R^n$ is the set on which $f$ is differentiable.
        % whole complement is negligible~\cite[Theorem~9.60]{rockafellar2009variational}.
        % and $\by\xxrightarrow[\mathbb{D}]{\vphantom{}}\bx$ means $\by\rightarrow\bx$ with $\by\in\mathbb{D}$.
        % $$
        %     \partial_C f(\bx):=\{\bs\in\R^n:\exists\,\by\rightarrow\bx~\text{with}~\nabla f(\by)~\text{exists and converges to}~\bs\}.
        % $$
    \end{itemize}
    \end{definition}

    For more properties and relationships of the above three subdifferentials (and beyond), we refer interested readers to~\cite[Section~8]{rockafellar2009variational},~\cite[Section~2]{clarke1990optimization},~\cite[Section~4.3]{cui2021modern}, and~\cite{li2020understanding}. In addition, we would like to recommend the following references~\cite{tian2021hardness,tian2022computing,tian2022finite,tian2023testing,tian2024no} that are helpful for strengthening the understanding of these concepts.

% \begin{theorem}
%     A locally Lipschitz function is subdifferentially regular if and only if it is Clarke regular.
% \end{theorem}

% \begin{proof}
    
% \end{proof}

Given a composite function $f:=g\circ\HI$ where $\HI:\R^n\rightarrow\R^m$ is strictly differentiable and $g:\R^m\rightarrow\R$ is locally Lipschitz, the subdifferential chain rule pertains to the relationship between the following two sets $\partial_{\triangleleft} f(\bx)$ and $\bJ^\T_{\HI}(\bx) \partial_{\triangleleft} g(\HI(\bx))$ at some interested point $\bx\in\R^n$, where $\partial_{\triangleleft}$ can be either $\widehat{\partial}$, $\partial$, or $\partial_C$ (cf.~Definition~\ref{def:subdiffs}). In this paper, we mainly focus on the Clarke subdifferential chain rule, i.e., $\partial_{\triangleleft}=\partial_C$. It is well-known that under these assumptions,
% under the assumptions of $\HI$ and $g$,
% provided $\HI$ is strictly differentiable at $\bx$,
% and $g$ is Lipschitz continuous near $\HI(\bx)$, 
% which is the case for all examples to be studied, 
we always have the following set inclusion $\partial_C f(\bx)\subseteq\bJ^\T_{\HI}(\bx) \partial_C g(\HI(\bx))$; see~\cite[Theorem~2.3.10]{clarke1990optimization}. Therefore, what we are really interested in are the circumstances under which the above inclusion actually holds as an equality. 
% (At this position, we would like to declare that, 
(In the sequel, for simplicity, we abbreviate the above event as ``the subdifferential chain rule holds.'') 
% or any of its synonyms, we actually mean that the chain rule of the Clarke subdifferential holds as an equality; such a terminology is only for succinctness of presentation.) 
The local surjection property~\cite{ioffe1987local} as defined below is a critically important concept for our purposes as it suffices to guarantee the validity of the subdifferential chain rule\footnote{We remark that such a property may not be sufficient for the Fr\'echet and limiting subdifferential chain rules.}~\cite[Theorem~2.3.10]{clarke1990optimization} and is much weaker than surjectivity, as evidenced by the inverse function theorem.
% Actually, the main contribution of this paper is the establishments of this interesting property for various examples:

\begin{definition}\label{def:lsp}
    The mapping $\F:\R^n\rightarrow\R^m$ is said to have the local surjection property at $\bx\in\R^n$ if for any $t>0$, there exists some $r>0$, such that $\F(\bx)+r\bbB\subseteq\F(\bx+t\bbB)$.
\end{definition}

We remark that apart from our purposes,
% although we will not pursue here, 
the local surjection property also plays an important role in second-order variational analysis (see, e.g.,~\cite[Theorem~2.33]{mordukhovich2024second}) and differentiable manifolds (see, e.g.,~\cite{lee2012smooth}). Besides, it also has close connections to various properties such as the Aubin property, metric regularity, and coderivative nonsingularity; see, e.g.,~\cite[Theorem~9.43]{rockafellar2009variational}.

% We remark that the local surjection property can also be regarded as a relaxation of the linear openness studied in variational analysis~\cite[Theorem~9.43(c)]{rockafellar2009variational}.
% % , where the requirement $t=\Omega(r)$ is demanded. 
% However, due to the complicated structures of the examples to be studied, existing conditions for guaranteeing the linear openness, such as~\cite[Example~9.45]{rockafellar2009variational} requiring the graph of $\F$ to be convex, are generally invalid, which necessities a different approach for study this property.

    \subsection{Prototypical examples}\label{sec:examples}
        In this subsection, we introduce in detail the prototypical examples mentioned in Section~\ref{sec:introduction}.
        % whose local surjection property will be further studied. As we shall see later, these examples are very representative and have close connections to a variety of neural models.
        % We remark that in all subsequent introductions we do not consider any regularization.
        % % on the corresponding objective functions. 
        % This is because 
        % % (recall the goal of this paper) 
        % for smooth regularizers such as the Frobenius norm squared one, by~\cite[Corollary~2.3.1]{rockafellar2009variational} we trivially know the subdifferential chain rule holds for the regularized objective function whenever it does so for the unregularized one, and thus it suffices to study the latter, while for nonsmooth regularizers the analysis becomes much more trickier,
        % % by the techniques to be developed, 
        % and we thus leave them as future works.
        % % The same applies to the FM model below.

    \subsubsection{Matrix factorization and tensor CP factorization}\label{sec:mf-cp}
    Matrix factorization~\cite{haeffele2019structured} is a fundamental technique for a variety of modern data analytics tasks such as signal processing~\cite{fu2016power}, network analytics~\cite{bai2024dual}, and recommender systems~\cite{zeng2023msbpr}. It also serves as a building block for many other more sophisticated learning paradigms including low-rank matrix recovery~\cite{li2020nonconvex} and deep linear networks~\cite{yaras2023law}. Given $\bZ\in\R^{m\times n}$, MF aims at decomposing it into two factors $\bX\in\R^{d\times m}$ and $\bY\in\R^{d\times n}$ for some $d\in\N$ such that $\bZ\approx\bX^{\T}\bY$. Most methods for finding such a factorization resort to
    % people usually resort to minimize the following 
    optimizing some objective function $\ell(\bX^{\T}\bY;\bZ)$,
% $\sum_{i=1}^m\sum_{j=1}^n\ell(z_{i j},\bx_i^{\T}\by_j)$, 
where $\ell(\bullet;\bZ):\mathbb{R}^{m\times n}\rightarrow\mathbb{R}$ encodes the error between $\bZ$ and the factorization (e.g., $\ell(\bullet;\bZ)=\|\bullet-\bZ\|_{\TF}^2$).
% , then we recover the standard Frobenius norm squared loss. 
% We would also like to remark that, actually, the designation of $\ell_{\bZ}$ can be rather flexible.) 

In a similar (but generalized) vein, the tensor CP factorization~\cite{kolda2009tensor} intends to approximate some tensor $\mathbfcal{T}\in\R^{n_1\times n_2\times n_3}$ by $\llbracket\bX,\bY,\bZ\rrbracket\in\R^{n_1\times n_2\times n_3}$, where $\bX\in\R^{d\times n_1},\bY\in\R^{d\times n_2},\bZ\in\R^{d\times n_3}$, and
$$
(\llbracket\bX,\bY,\bZ\rrbracket)_{i j k}:=\langle\bx_i,\by_j,\bz_k\rangle,\quad\text{for all $(i,j,k)\in[n_1]\times[n_2]\times[n_3]$}.
$$ 
Similarly, a general approach for finding such factors is to minimize $\ell(\llbracket\bX,\bY,\bZ\rrbracket;\mathbfcal{T})$.
% $\sum_{i=1}^{n_1}\sum_{j=1}^{n_2}\allowbreak\sum_{k=1}^{n_3}\allowbreak\ell(t_{i j k},\langle\bx_i,\by_j,\bz_k\rangle)$.

    \subsubsection{Factorization machine}\label{sec:fm}
    Factorization machine~\cite{rendle2010factorization} is a new type of simple, efficient, yet powerful second-order supervised predictor. Given a data vector $\boldsymbol{x}\in\mathbb{R}^{d_0}$ for some $d_0\in\N$, FM can be formally expressed as
\begin{equation}
\label{eq:c-fm}
    f_{\textnormal{FM}}(\boldsymbol{x} ;\boldsymbol{w}, \boldsymbol{P}):=\boldsymbol{w}^{\T}\boldsymbol{x}+\sum_{i=1}^{d_0-1}\sum_{j=i+1}^{d_0}(\boldsymbol{P}^{\T}\boldsymbol{P})_{i j} \cdot x_{i} x_{j},
\end{equation}
where $\boldsymbol{w}\in\mathbb{R}^{d_0}$ and $\boldsymbol{P}\in\mathbb{R}^{d\times d_0}$ are model parameters. One can clearly recognize that the weights of second-order feature interactions, namely $(\boldsymbol{P}^{\T}\boldsymbol{P})_{i j}$'s, are in a factorized manner, and this is the key for FMs to accelerate their computation and optimization, save their memory usages, and improve their generalization performance. 
% In what follows, we will hide the parameters in $f_{\textnormal{FM}}(\boldsymbol{x} ;\boldsymbol{w}, \boldsymbol{P})$ for succinctness. 
It is well-known that $f_{\textnormal{FM}}(\boldsymbol{x} ;\boldsymbol{w}, \boldsymbol{P})$ is multiaffine in the columns of $\bP$; see, e.g.~\cite[Lemma~2]{blondel2016polynomial}.
Given a dataset $\{(\boldsymbol{x}_i,y_i)\}_{i=1}^n\subseteq\R^{d_0}\times\R$ with $n\in\N$ data samples, the objective function for training an FM is typically given by $\ell(\bigvee_{i=1}^n f_{\textnormal{FM}}(\boldsymbol{x}_i;\boldsymbol{w}, \boldsymbol{P});\by)$.
% $\sum_{i=1}^{n}\ell(y_{i},f_{\textnormal{FM}}(\boldsymbol{x}_i))$.
% $$
% \label{eq:general-obj}
%     \frac{1}{n}\sum_{i=1}^{n}\ell(y_{i},f_{\textnormal{FM}}(\boldsymbol{x}^{(i)})),
% $$
% where $\ell:\mathbb{R}\times\mathbb{R}\rightarrow\mathbb{R}$ is some arbitrary but locally Lipschitz loss function.
In this paper, as in~\cite{blondel2016polynomial}, we mainly focus on the homogeneous FM that excludes the linear term $\boldsymbol{w}^{\T}\boldsymbol{x}$.
% This is almost without loss of generality, since we can always use the simple data augmentation trick introduced in~\cite[Section~8]{blondel2016polynomial} and~\cite[Section~5.1]{blondel2016higher}
% % (see also the discussions therein) 
% to include such a term under the homogeneous framework, except introducing some additional parameter sharing that turns out to be actually favorable
% % , and this strategy is also known to work very well in
% in practice~\cite[Section~5.1]{blondel2016higher}. 
For a comprehensive treatment of the theoretical foundations of FMs, we refer interested readers to~\cite{blondel2015convex,blondel2016polynomial,blondel2016higher,blondel2017multi,atarashi2021factorization}.

    \section{Local surjection property and subdifferential chain rule}\label{sec:main-results}
    In this section, we study the local surjection property of the examples mentioned in Section~\ref{sec:examples} and its implications for subdifferential chain rules. We begin with a negative observation.
    
    \subsection{A failure of the subdifferential chain rule of matrix factorization}
    % The first question that is natural to ask is
    It is very natural to ask if the subdifferential chain rules of the aforementioned examples actually always hold at all. The following example, which is constructed for the MF problem (and its adaption to the FM problem is straightforward), answers the above question in the negative.
    
    \begin{example}\label{ex:negative}
        Consider the composite function 
        % $f(x_1,x_2,y_1,y_2):=h(g(x_1,x_2,y_1,y_2))$
        $f:=g\circ \HI$ from $\R^4$ to $\R$, where $\HI:\R^4\rightarrow \R^4$ with
        % \footnote{For presentation convenience, all vectors within this example are row vectors.}:
        $$
        \HI(x_1,x_2,y_1,y_2):=(x_1 y_1, x_2 y_1, x_1 y_2, x_2 y_2)^{\T},
        $$
        and $g:\R^4\rightarrow\R$ with
        $$
        g(z_1, z_2, z_3, z_4):=\sigma(z_1 z_4)-\sigma(z_2 z_3),\quad\text{where $\sigma:=\max\{\bullet-1,0\}$}.
        $$
        (We remark that by~\cite[Theorem~9.7,~Exercise~9.8(b-c)]{rockafellar2009variational}, $g$ is locally Lipschitz.) It is easy to notice that $f$ is actually the zero function, and thus by~\cite[Exercise~8.8(b)]{rockafellar2009variational} and~\cite[Proposition~4.3.2(b)]{cui2021modern} we know $\partial_C f(x_1,x_2,y_1,y_2)=\{\bd{0}\}$ for any $(x_1,x_2,y_1,y_2)^{\T}\in\R^4$. On the other hand, by using~\cite[Proposition~2.5]{rockafellar1985extensions}, the Clarke regularity of $\sigma$~\cite[Proposition~2.3.6(b)]{clarke1990optimization}, the strict differentiability of $(x,y)\mapsto x y$~\cite[Corollary~2.2.1]{clarke1990optimization}, and~\cite[Proposition~2.3.1,~Theorem~2.3.10]{clarke1990optimization}, we can show
        $$
        \partial_C g(z_1, z_2, z_3, z_4)=\{(x z_4,-y z_3,-y z_2,x z_1)^{\T}:x\in\partial_C\sigma(z_1 z_4),\,y\in\partial_C\sigma(z_2 z_3)\}.
        $$
        However, this implies
        % for $\overline{\bw}:=(\overline{x}_1,\overline{x}_2,\overline{y}_1,\overline{y}_2)=\bd{1}$, 
        $$
        \bJ^\T_{\HI}(\bd{1}) \partial_C g(\HI(\bd{1}))=\begin{pmatrix}1& 0& 1& 0 \\ 0& 1& 0& 1 \\ 1& 1& 0& 0 \\ 0& 0& 1& 1
        \end{pmatrix}\cdot\left\{\begin{pmatrix}x \\ -y \\ -y \\ x\end{pmatrix}:x,y\in[0,1]\right\}=\left\{(x-y)\cdot\bd{1}:x,y\in[0,1]\right\}\neq\{\bd{0}\},
        $$
        and thus the subdifferential chain rule of $g\circ\HI$ fails at $\bd{1}$.
    \end{example}

    Although Example~\ref{ex:negative} delivers some negative message, it also confirms the necessity and importance to study the subdifferential chain rules for these problems. Interestingly, we will soon see that as long as each scalar variable of $f$ in Example~\ref{ex:negative} (such as $x_1$) is replaced by some vector variable with moderate dimension, the corresponding subdifferential chain rule will hold everywhere.

\subsection{Local surjection property, subdifferential chain rule of matrix factorization}\label{sec:lsp-mf}

In this subsection, we study the local surjection property of the following mapping
\begin{equation}\label{eq:mapping-MF}
\HI_{\operatorname{MF}}(\bX,\bY):=\bX^{\T}\bY,\quad\text{from $\R^{d\times m}\times \R^{d\times n}$ to $\R^{m\times n}$},
\end{equation}
% \begin{equation}\label{eq:mapping-MF}
% \begin{aligned}
% \HI_{\operatorname{MF}}:(\R^{d})^{m+n}&\rightarrow\R^{m n}\\
% (\{\bx_i\}_{i=1}^m,\{\by_j\}_{j=1}^n)&\mapsto\operatorname{vec}\left(\bigvee_{i,j=1}^{m,n}\bx_i^{\T}\by_j\right),
% \end{aligned}
% \end{equation}
where $m,n,d\in\N$ are arbitrary. Clearly, this mapping corresponds to the MF example introduced in Section~\ref{sec:mf-cp}.
% It is worth noting that by defining: 
% $$
% \bx:=\bigvee_{i=1}^m \bx_i,\quad\by:=\bigvee_{i=1}^m \by_i,\quad\bA_{i j}:=\left(\be^m_i\otimes\be^n_j\right)\boxtimes\bI_d,
% $$
% we can regard (\ref{eq:mapping-MF}) as a special case of the general bilinear mapping introduced in Section~\ref{sec:lsp-general}. 
To begin, we give a sufficient condition for the local surjection property of $\HI_{\operatorname{MF}}$ (\ref{eq:mapping-MF}) at the origin that depends on the latent dimension $d$ only.
% pertaining to the embedding dimension for the simplest case.
% We remark that such a condition shares a similar spirit with~\cite[Theorem~6]{atarashi2021factorization}.

\begin{proposition}\label{prop:surj-mf}
    % Suppose $\bbA:=\{(\be^m_i\otimes\be^n_j)\boxtimes\bI_d\}_{i,j=1}^{m,n}$ for some
    % Let $m,n,d\in\N$. Then, p
    Provided $d\ge\min\{m,n\}$, 
    % the mapping 
    the mapping $\HI_{\operatorname{MF}}$ (\ref{eq:mapping-MF})
    % $\HI_{\operatorname{MF}}(\{\bx_i\}_{i=1}^m,\{\by_j\}_{j=1}^n)\allowbreak:=\operatorname{vec}(\bigvee_{i,j=1}^{m,n}\bx_i^{\T}\by_j)$ from $(\R^{d})^{m+n}$ to $\R^{m n}$ 
    is locally surjective at the origin.
\end{proposition}

\begin{proof}
    For any $t>0$, let us consider some point $(\bX,\bY)\in t\bbB$ that is subject to change. 
    % Let us first evenly partition $\bx:=\bigvee_{i=1}^m \bx_i$ and $\by:=\bigvee_{i=1}^m \by_i$.
    % For ease of presentation, we also let $\bX:=(\bx_1,\bx_2,\ldots,\bx_m)\in\R^{d\times m}$, 
    % Then, it follows that $\HI_{\operatorname{MF}}(\bx,\by)=\bigvee_{i,j=1}^{m,n}\bx_i^{\T}\by_j$.
    Due to symmetry, we assume w.l.o.g.\ that $m\ge n$. Then, by fixing $\by_j:=(t/\sqrt{2n})\cdot\be_j$, which is legitimate due to the assumption of $d$, we have
    % \footnote{Recall by our notations introduced in Section~\ref{sec:notations}, we implicitly have defined $\bX:=(\bx_1,\bx_2,\ldots,\bx_m)\in\R^{d\times m}$.}:
    $$
    \HI_{\operatorname{MF}}(t\bbB)\supseteq\left\{(t/\sqrt{2n})\cdot\bX^{\T}\cdot(\bI_n\vee\bO_{d-n,n}):\|\bX\|_{\TF}\le t/\sqrt{2}\right\}.
    $$ 
    Thus, by fixing the last $d-n$ coordinates of $\bx_i$ to be zero for each $i\in[m]$, varying the remaining coordinates of each $\bx_i$ over $(t/\sqrt{2m})\cdot\bbB$,
    % $(\bO_{d-n,n},\bI_{d-n})\cdot\bx_i=\bd{0}$ 
    and using the easily-verifiable fact that
    \begin{equation}\label{eq:easy-fact}
    \text{for any $k\in\N$, $\bz_1,\ldots,\bz_k\in\R^n$, and $r>0$, it always holds that $\bigvee_{i=1}^k\bz_i+r\bbB\subseteq\prod_{i=1}^k(\bz_i+r\bbB)$},
    \end{equation}
    we can conclude $\HI_{\operatorname{MF}}(t\bbB)\supseteq(t^2/\sqrt{4mn})\cdot\bbB$, which completes the proof.
\end{proof}

Despite its simplicity, as a direct application, we can obtain the following interesting corollary concerning the three subdifferentials in Definition~\ref{def:subdiffs} of any MF-like problem at the origin.
% and as we shall soon see, it would further hint us something that has very important consequences in optimization (but to our best knowledge has also been overlooked in the literature):

% As an application, let us consider the matrix factorization problem: Given a target matrix $\bA\in\R^{m\times n}$, we would like to factorize it approximately as $\bA\approx\bX^{\T}\bY$. A commonly used strategy to find such a factorization is to solve the following optimization problem $\min_{\bX,\bY}g(\operatorname{vec}(\bX^{\T}\bY))$, where $g:\R^{m\times n}\rightarrow\R$ is some loss function, such as the least-squares loss $\|\cdot-\operatorname{vec}(\bA)\|_{\TF}^2$.

\begin{corollary}\label{cor:mf-origin}
    Suppose that $\ell:\R^{m\times n}\rightarrow\R$ is locally Lipschitz and $d\ge\min\{m,n\}$, and let 
    \begin{equation}\label{eq:composition-function-MF}
    f(\bX,\bY):=\ell(\bX^{\T}\bY),\quad\text{from $\R^{d\times m}\times \R^{d\times n}$ to $\R$}.
    \end{equation}
    % , i.e., the problem is overparameterized. 
    Then, it holds that\footnote{Here, we have identified $\R^{d\times m}\times \R^{d\times n}$ with $\R^{d\times (m+n)}$ to avoid notational redundancy (otherwise we should write, e.g., $\partial_C f(\bO,\bO)$). Such a convention will be often adopted in the sequel as long as there is no ambiguity.}
    $$
        \partial_C f(\bO)=\partial f(\bO)=\widehat{\partial} f(\bO)=\{\bO\},
    $$
    and $f$ is strictly differentiable at the origin with $\nabla f(\bO)=\bO$.
\end{corollary}

\begin{proof}
    % Let us denote $\bx:=\operatorname{vec}(\bX)$ and $\by$ similarly. Then, it holds that $f(\bX,\bY)=$
    We first observe $f=\ell\circ\HI_{\operatorname{MF}}$,
    % for the $\HI_{\operatorname{MF}}$ defined in (\ref{eq:mapping-MF}), 
    and thus by Proposition~\ref{prop:surj-mf}, the assumption of $d$, and~\cite[Theorem~2.3.10]{clarke1990optimization}, we know the subdifferential chain rule holds for this composition at the origin. Since the Jacobian of $\HI_{\operatorname{MF}}$ at the origin is zero, the characterization of $\partial_C f(\bO)$ directly follows from the nonemptiness of Clarke subdifferentials~\cite[Proposition~2.1.2(a)]{clarke1990optimization} (see also~\cite[Proposition~4.3.1(d)]{cui2021modern}).
    % $$
    % \partial_C f(\bX,\bY)|_{\bX=\bO,\bY=\bO}=\begin{pmatrix}
    %     \bA_1\by & \bA_2\by & \cdots & \bA_{mn}\by \\
    %     \bA_1^{\T}\bx & \bA_2^{\T}\bx & \cdots & \bA_{mn}^{\T}\bx
    % \end{pmatrix}\Bigr|_{\bX=\bO,\bY=\bO}\cdot\partial_C g(\boldsymbol{0})=\{\bO\},
    % $$
    % as desired, where $\bA_i$ is defined to be the $i$th element in $\mathbb{A}$. 
    % This, together with~\cite[Proposition~4.3.4(ii)]{cui2021modern}, further implies that
    % Given this, by using the fact that $\widehat{\partial}f(\bX,\bY)\subseteq\partial f(\bX,\bY)\subseteq \partial_C f(\bX,\bY)$ for any $\bX,\bY$~\cite[Proposition~4.3.2]{cui2021modern} and the nonemptiness of limiting subdifferentials~\cite[Theorem~9.13]{rockafellar2009variational} (see also~\cite[Proposition~4.3.1(c)]{cui2021modern}), 
    This, together with~\cite[Proposition~4.3.4(ii)]{cui2021modern}, further implies the remaining results, as desired.
\end{proof}

% As we can see, the interested subdifferentials all admit a very simple structure.
Inspired by Corollary~\ref{cor:mf-origin}, we realize it is possible to establish something even stronger.

\begin{corollary}\label{cor:zero-subdiff}
    Corollary~\ref{cor:mf-origin} still holds in absence of the condition $d\ge\min\{m,n\}$.
\end{corollary}

\begin{proof}
    % The above corollary actually admits another much simpler proof as follows, which could even work in \textnormal{absence} of the condition $d\ge\min\{m,n\}$: 
    By~\cite[Theorem~2.3.10]{clarke1990optimization} and the nonemptiness of Clarke subdifferentials, we know $\partial_C f(\bO)\subseteq\{\bO\}$. However, by (again) the nonemptiness of $\partial_C f(\bO)$, we know it must be $\{\bO\}$. The remaining results can be shown by invoking~\cite[Proposition~4.3.4(ii)]{cui2021modern} again. This completes the proof.
\end{proof}

% We remark that despite its simplicity, unfortunately, the proof of Corollary~\ref{cor:zero-subdiff} is almost impossible to be generalized beyond the origin.
Recall that points at which the subdifferentials include zero have special significance from an optimization perspective. In fact, we would like to make the following important observation.

% \begin{remark}
%     % This result may seem incredible, but one will soon convince its reasonability after examining the example $g=\|\cdot-\bA\|_{\TF}^2$. 
%     Remark~\ref{rmk:zero-subdiff} caveats us when studying these MF-like problems, unless one can show that all other sharper stationarity concepts are not oracle-polynomial-time computable, we should put our target for optimization at least beyond the Fr\'echet stationarity (which is the sharpest possible among the three aforementioned), as such a stationary point can always be obtained even without calling any oracle and doing any computation by trivially returning the origin.
%     % We would like to stress that this perspective seems unrecognized in the literature.
%     This observation still holds for all remaining examples in this paper and even beyond.
% \end{remark}

\begin{remark}\label{rmk:caveat}
    Corollary~\ref{cor:zero-subdiff} caveats us that when dealing with problems that can be formulated as (\ref{eq:composition-function-MF}), no matter how bizarre the loss function $\ell$ looks like, as long as it is locally Lipschitz, we should place our goal for optimization beyond the stationarity concept
    % \footnote{To be more precise for this stationarity concept, we mean some point at which the subdifferential contains zero.}
    induced by any meaningful subdifferential construction.
    % any subdifferential construction that collapses to a singleton set consisting of the gradient evaluated there at points where $f$ is strictly differentiable, including all the subdifferentials introduced in Definition~\ref{def:subdiffs} and should also encompass all other meaningful constructions of subdifferentials. 
    This is because such a stationary point can always be obtained even without calling any oracle and performing any computation by trivially returning the origin due to Corollary~\ref{cor:zero-subdiff}, and it also hints us that for these problems, we should be able to do something better.
    % than finding a stationary point only.
    (This also corroborates some recent advances on the convergence of gradient algorithms for standard MF~\cite{du2018algorithmic,ye2021global,liu2021noisy}, showing even a global minimizer can be obtained in polynomial time.) 
    On the other hand, Corollary~\ref{cor:zero-subdiff} also acknowledges the ubiquitous existence of spurious stationarity in this type of problems, since in almost all cases the origin will not be a desired solution as it carries no information at all.
    % We would like to stress that this perspective seems unrecognized in the literature.
    % We highlight that this observation also holds for all remaining examples in this paper and even beyond.
    We would also like to stress that Corollary~\ref{cor:zero-subdiff} and its implications also hold for their symmetric and higher-order counterparts and even beyond,
    % , manifold-constrained\footnote{To see this, it suffices to examine the equivalent definition of the Riemannian gradient~\cite[Proposition~3.61]{boumal2023introduction} as well as the first-order optimality condition~\cite[Proposition~4.6]{boumal2023introduction}.} counterparts, their combinations, their smoothly regularized versions, and even beyond, 
    % driving us to rethink about the meaningfulness 
    shedding new light on the usefulness
    of the convergence guarantees of many existing algorithms, such as~\cite[Theorem~5(a)]{hastie2015matrix},~\cite[Theorem~4.2]{li2015convergence},~\cite[Theorem~1]{wang2017polynomial},~\cite[Theorem~2]{lin2017robust}, and~\cite[Corollaries~1-2]{zhu2018dropping}.
\end{remark}

% \begin{proposition}
%     \sout{Suppose $\bbA:=\{(\be^m_i\otimes\be^n_j)\boxtimes\bI_d\}_{i,j=1}^{m,n}$ for some $m,n,d\in\N$ with $d\ge\min\{m,n\}$. Then the mapping $\HI_{\operatorname{MF}}(\bx,\by):=\bigvee_{\bA\in\bbA}\bx^{\T}\bA\by$ has the local surjection property at any point $(\bx,\by)\in\{(\bx,\by):\min_{i\in[m]}\|\bx_i\|>0~\text{or}~\min_{j\in[n]}\|\by_j\|>0\}$, i.e., either $\bx$ or $\by$ has no zero subvector.}
% \end{proposition}

% \begin{proof}
%     \sout{W.l.o.g., we may assume $\bx$ has no zero subvector. We first proceed similarly as in the proof of Proposition~\ref{prop:surj-mf}, and arrive at $\HI_{\operatorname{MF}}(\bx,\by)=\bigvee_{i,j=1}^{m,n}\bx_i^{\T}\by_j$. Afterwards, let us consider $(\bx^\prime,\by^\prime)\in(\bx,\by)+t\bbB$. By fixing $\bx^\prime=\bx$ and letting $\by^\prime_j$ vary in $\by_j+(t/\sqrt{n})\cdot\bbB$ for each $j\in[n]$, we arrive at $\HI_{\operatorname{MF}}((\bx,\by)+t\bbB)\supseteq\prod_{i,j=1}^{m,n}(\bx_i^{\T}\by_j+(t/\sqrt{n})\cdot\bx_i^{\T}\bbB)$. By~\cite[Problem~5.1.3]{barvinok2002course}, we know $\bx_i^{\T}\bbB$ is an ellipsoid in $\R$ (and thus an interval) centered at the origin. With this observation, it is easy to see that $\bx_i^{\T}\bbB=[-\|\bx_i\|,\|\bx_i\|]$ actually, and thus, using again the fact that $\bigvee_{i=1}^k\bz_i+r\bbB\subseteq\prod_{i=1}^k(\bz_i+r\bbB)$, we conclude that $\HI_{\operatorname{MF}}((\bx,\by)+t\bbB)\supseteq\bigvee_{i,j=1}^{m,n}\bx_i^{\T}\by_j+(t\min_{i\in[m]}\|\bx_i\|/\sqrt{n})\cdot\bbB$, as desired.}
% \end{proof}

Recall that Proposition~\ref{prop:surj-mf} concerns the local surjection property of $\HI_{\operatorname{MF}}$ (\ref{eq:mapping-MF}) only at the origin. Our next step is to generalize such a result to some arbitrary point.

\begin{theorem}\label{thm:surj-mf-anypoint}
    % Suppose $\bbA:=\{(\be^m_i\otimes\be^n_j)\boxtimes\bI_d\}_{i,j=1}^{m,n}$ for some 
    % Let $m,n,d\in\N$. 
    Given some interested point $(\bX,\bY)\in\R^{d\times m}\times \R^{d\times n}$, provided
    \begin{equation}\label{eq:surj-mf-anypoint}
    d\ge\operatorname{dim}(\operatorname{span}(\{\bx_i\}_{i=1}^m\cup\{\by_j\}_{j=1}^n))+\min\{m,n\},
    \end{equation}
    the mapping $\HI_{\operatorname{MF}}$ (\ref{eq:mapping-MF})
    % $\HI_{\operatorname{MF}}(\{\bx_i\}_{i=1}^m,\{\by_j\}_{j=1}^n):=\operatorname{vec}(\bigvee_{i,j=1}^{m,n}\bx_i^{\T}\by_j)$ from $(\R^{d})^{m+n}$ to $\R^{m n}$
    is locally surjective at this point. In particular, provided
    \begin{equation}\label{eq:surj-mf-everywhere}
    d\ge m+n+\min\{m,n\},
    \end{equation}
    the mapping $\HI_{\operatorname{MF}}$ (\ref{eq:mapping-MF}) is locally surjective everywhere.
\end{theorem}

\begin{proof}
    Similar to the proof of Proposition~\ref{prop:surj-mf},
    % , and arrive at $\HI_{\operatorname{MF}}(\bx,\by)=\bigvee_{i,j=1}^{m,n}\bx_i^{\T}\by_j$. Afterwards, 
    let us consider the following point
    % in $(\bx,\by)+t\bbB$. It is easy to see that 
    $$
    (\bX^{\prime},\bY^{\prime}):=(\bX+\varepsilon\cdot\bW,\bY+\varepsilon\cdot\bZ)\in(\bX,\bY)+t\bbB,
    $$
    where $\bw_1,\ldots,\bw_m,\bz_1,\ldots,\bz_n\in\bbB$ are independent and subject to change, $\varepsilon:=\min\{t/\sqrt{2m}, t/\sqrt{2n}\}$, and the last inclusion can be verified easily. It is direct to see that with this parametrization 
    $$
    \HI_{\operatorname{MF}}(\bX^{\prime},\bY^{\prime})=\bX^{\T}\bY+\varepsilon\cdot\bX^{\T}\bZ+\varepsilon\cdot\bW^{\T}\bY+\varepsilon^2\cdot\bW^{\T}\bZ.
    % \bigvee_{i,j=1}^{m,n}\left(\bx_i^{\T}\by_j+\varepsilon\cdot\bx_i^{\T}\bz_j+\varepsilon\cdot\bw_i^{\T}\by_j+\varepsilon^2\cdot\bw_i^{\T}\bz_j\right).
    $$
    Next, let us consider the subspace
    $$
    \bbV:=\operatorname{span}(\{\bx_i\}_{i=1}^m\cup\{\by_j\}_{j=1}^n).
    $$
    It is immediate that $\operatorname{dim}(\bbV)\le m+n$. Besides, by restricting each $\bw_i$ and $\bz_j$ to $\TP_{\bbV^\perp}(\bbB)=\bbB\cap\bbV^{\perp}$, which is a unit ball as well with dimension $d-\operatorname{dim}(\bbV)>0$ as assumed, we can eliminate all the cross terms above (i.e., $\bX^{\T}\bZ$ and $\bW^{\T}\bY$),
    % in the expression of $\HI_{\operatorname{MF}}(\bx^\prime,\by^\prime)$, 
    which results in
    \begin{equation}\label{eq:MF-temp1}
    \HI_{\operatorname{MF}}((\bX,\bY)+t\bbB)\supseteq\bX^{\T}\bY+\left\{\varepsilon^2\cdot\bW^{\T}\bZ:\bw_i\in\TP_{\bbV^\perp}(\bbB),\,\forall\,i\in[m],\,\bz_j\in\TP_{\bbV^\perp}(\bbB),\,\forall\, j\in[n]\right\}.
    \end{equation}
    In the following, because we only care about the inner products between $\bw_i$'s and $\bz_j$'s, which are coordinate independent, we may assume w.l.o.g.\ that $\TP_{\bbV^\perp}(\bbB)$ is aligned with the standard bases, and henceforth $\bw_1,\ldots,\bw_m,\bz_1,\ldots,\bz_n\in\R^{d-\operatorname{dim}(\bbV)}$. To conclude the proof, it suffices to apply Proposition~\ref{prop:surj-mf} to the last set in (\ref{eq:MF-temp1}), which is legitimate since $d-\operatorname{dim}(\bbV)\ge\min\{m,n\}$.  
    % Let $\bZ:=(\bz_1,\bz_2,\ldots,\bz_n)\in\R^{(d-\operatorname{dim}(\bbV))\times n}$. Besides, we also assume w.l.o.g.\ that $m\ge n$ as before. To conclude the result, we only need to first simply let $\bw_i:=\be^{d-\operatorname{dim}(\bbV)}_i$ for each $i\in[m]$, which is doable due to the assumption $d\ge\operatorname{dim}(\bbV)+m$ and gives up to $\bigvee_{i,j=1}^{m,n}\varepsilon^2\cdot\bw_i^{\T}\bz_j=\varepsilon^2\cdot\operatorname{vec}((\bI_m,\bO_{m,d-\operatorname{dim}(\bbV)-m})\cdot\bZ)$, and then apply a similar argument used at the end of the proof of Proposition~\ref{prop:surj-mf}.
\end{proof}

We remark that the above theorem covers Proposition~\ref{prop:surj-mf} as a special case when the interested point is the origin. Besides, the following corollary is also immediate.

\begin{corollary}\label{cor:mf-chain-rule}
    % Suppose $\ell:\R^{m\times n}\rightarrow\R$ is locally Lipschitz and $d\ge m+n+\min\{m,n\}$, and let $f(\bX,\bY):=\ell(\bX^{\T}\bY)$ with domain $\R^{d\times m}\times \R^{d\times n}$.
    % , i.e., the problem is overparameterized. 
    Provided (\ref{eq:surj-mf-everywhere}) holds,
    % $d\ge m+n+\min\{m,n\}$, 
    the subdifferential chain rule of (\ref{eq:composition-function-MF}) holds everywhere.
\end{corollary}

\begin{proof}
This directly follows from the combination of Theorem~\ref{thm:surj-mf-anypoint} and~\cite[Theorem~2.3.10]{clarke1990optimization}.
\end{proof}

Another remark on some specializations of Theorem~\ref{thm:surj-mf-anypoint} is also in order.

\begin{remark}\label{rmk:special-case}
    It is easy 
    % to adapt the reasoning used in the proof of Theorem~\ref{thm:surj-mf-anypoint} 
    to see that the conditions in Theorem~\ref{thm:surj-mf-anypoint} are also sufficient for the local surjection property of each submapping
    $$
    \HI_{\operatorname{MF}}^{\mathbb{D}}(\bX,\bY):=\bigvee_{(i,j)\in\mathbb{D}}\bx_i^{\T}\by_j,\quad\text{from $\R^{d\times m}\times \R^{d\times n}$ to $\R^{|\mathbb{D}|}$, where $\mathbb{D}\subseteq [m]\times[n]$ is arbitrary}.
    $$
    % $\HI_{\operatorname{MF}}_{\mathbb{D}}(\{\bp_i\}_{i=1}^{d_0}):=\bigvee_{(i,j)\in\mathbb{D}}(a_{i j}\cdot\bp_i^{\T}\bp_j)$
    % where $\mathbb{D}\subseteq\{(i,j):i\in[d_0-1],j\in[i+1,d_0]\cap\N\}$
    % and $a_{i,j}\neq 0$ for all $(i,j)\in\mathbb{D}$, 
    % and the proof goes almost verbatim. 
    Besides, this observation also applies to all other prototypical examples to be studied later.
    % , and we will not mention it anymore for succinctness.
\end{remark}

At this position, it is very natural to ask if the mapping $\HI_{\operatorname{MF}}$ (\ref{eq:mapping-MF}) actually possesses the local surjection property for all $m,n,d\in\N$ and interested points at all. The following example, whose construction is largely inspired by the computation of the Vapnik-Chervonenkis dimension of linear classifiers in $\R^2$ (see, e.g.,~\cite[Example~8.3.3]{vershynin2018high}), answers the above question in the negative.

\begin{example}\label{ex:mf1}
    Consider the case where $m,n=2$ and $d=1$.
    % $$
    % \HI_{\operatorname{MF}}(\bx,\by):=\operatorname{vec}(\bx\by^{\T}),\quad\text{from $\R^2\times\R^2$ to $\R^4$}.
    % $$
    We manage to refute the local surjection property of $\HI_{\operatorname{MF}}$ (\ref{eq:mapping-MF}) at the origin by showing
    $$
    \operatorname{vec}(\HI_{\operatorname{MF}}(\R^{1\times 2}\times\R^{1\times 2}))\cap\operatorname{int}(\R^3_+\times\R_-)=\varnothing.
    $$
    To this end, suppose that there are some $\bx^{\star},\by^{\star}\in\R^2$ with $\operatorname{vec}(\bx^{\star}{\by^{\star}}^{\T})\in\operatorname{int}(\R^3_+\times\R_-)$. By unrolling the definition of the above inclusion, we know $x^{\star}_1,x^{\star}_2,y^{\star}_1,y^{\star}_2\neq 0$, and hence the pairs $(x^{\star}_1,y^{\star}_1)$, $(x^{\star}_1,y^{\star}_2)$, and $(x^{\star}_2,y^{\star}_1)$ all share the same sign.
    % , since, e.g., $x^{\star}_1 y^{\star}_1=a_1>0$, 
    This in turn implies the pair $(x^{\star}_2,y^{\star}_2)$ also shares the same sign, and thus $x^{\star}_2 y^{\star}_2>0$. However, from the assumption we also know $x^{\star}_2 y^{\star}_2<0$, which is a contradiction. 
    % The desired claim then follows from the arbitrariness of $\ba$. 
    As a consequence, since the Lebesgue measure of $r\bbB \cap\operatorname{int}(\R^3_+\times\R_-)$ is positive for any $r>0$ (actually, it is $1/16$ of the Lebesgue measure of $r\bbB$), $\HI_{\operatorname{MF}}(t\bbB)$ can never enclose any ball centered at the origin for any $t>0$, thus refuting its local surjection property there.
\end{example}

The above example together with Proposition~\ref{prop:surj-mf} exactly characterizes the phase transition of the local surjection property of $\HI_{\operatorname{MF}}$ (\ref{eq:mapping-MF})
% for the local surjection property of the aforementioned mapping at the origin
when $m,n=2$: Even though $d$ is only one smaller than $\min\{m,n\}$ in this case, the desired property no longer holds. It is then very interesting to ask if we can generalize the above example, especially its underlying philosophy, to higher dimensions. The answer is in the affirmative (despite its discovery was full of twists and turns), as shown below.

\begin{example}\label{ex:mf-general}
    Consider the case where $m=n\ge 2$ are arbitrary and $d=n-1$.
    % $$
    % \HI_{\operatorname{MF}}\pig((\bx_i)_{i=1}^n,(\by_j)_{j=1}^n\pig):=\operatorname{vec}\left(\bigvee_{i,j=1}^{n}\bx_i^{\T}\by_j\right),\quad\text{from $(\R^{n-1})^{2n}$ to $\R^{n^2}$}.
    % $$
    We manage to refute the local surjection property of $\HI_{\operatorname{MF}}$ (\ref{eq:mapping-MF}) at the origin as before by showing
    $$
    \HI_{\operatorname{MF}}(\R^{(n-1)\times n}\times\R^{(n-1)\times n})\cap\bbO=\varnothing,
    $$
    where
    $$
    \operatorname{vec}(\bbO):=\operatorname{relint}\left(\left(\prod_{i=1}^{n-2}(\R^2\times\{0\}^{i-1}\times\R_+\times\{0\}^{n-2-i})\right)\times\R_+^2\times\{0\}^{n-2}\times\R_+\times\R_-\times\{0\}^{n-2}\right).
    $$ 
    \begin{wrapstuff}[r,width=.35\textwidth,abovesep=-5pt]
    \begin{equation*}
\begin{pNiceArray}{cccc|cc}
\R & \R & \cdots & \R & {\R_+} & {\R_+} \\
\R & \R & \cdots & \R & {\R_+} & {\R_-} \\
\midrule
{\R_+} & {\{0\}} & \cdots & {\{0\}} & {\{0\}} & {\{0\}} \\
{\{0\}} & {\R_+} & {\cdots} & {\{0\}} & {\{0\}} & {\{0\}} \\
{\vdots} & {\vdots} & {\ddots} & {\vdots} & {\vdots} & {\vdots} \\
{\{0\}} & {\{0\}} & {\cdots} & {\R_+} & {\{0\}} & {\{0\}}
\end{pNiceArray}
\end{equation*}
\end{wrapstuff}
\noindent (For an intuitive understanding, a visualization of $\bbO$ before taking relative interior
% above by standardly ``folding'' its coordinates into the matrix space $\R^{n\times n}$
is given in the right panel.)
%     $\begin{pNiceArray}{ccc|cc}
% \R & \cdots & \R & {\R_+} & {\R_+} \\
% \R & \cdots & \R & {\R_+} & {\R_-} \\
% \midrule
% {\R_+} & {\{0\}} & {\{0\}} & {\{0\}} & {\{0\}} \\
% {\{0\}} & {\ddots} & {\{0\}} & {\vdots} & {\vdots} \\
% {\{0\}} & {\{0\}} & {\R_+} & {\{0\}} & {\{0\}} 
% \end{pNiceArray}$.)
% $$
% % \begin{pmatrix}
% % \R & \cdots & \R & \textcolor{red}{\R_+} & \textcolor{red}{\R_+} \\
% % \R & \cdots & \R & \textcolor{red}{\R_+} & \textcolor{red}{\R_-} \\
% % \textcolor{green}{\R_+} & \textcolor{green}{\{0\}} & \textcolor{green}{\{0\}} & \textcolor{blue}{\{0\}} & \textcolor{blue}{\{0\}} \\
% % \textcolor{green}{\{0\}} & \textcolor{green}{\ddots} & \textcolor{green}{\{0\}} & \textcolor{blue}{\vdots} & \textcolor{blue}{\vdots} \\
% % \textcolor{green}{\{0\}} & \textcolor{green}{\{0\}} & \textcolor{green}{\R_+} & \textcolor{blue}{\{0\}} & \textcolor{blue}{\{0\}} 
% % \end{pmatrix}.\text{)}
% \begin{pNiceArray}{ccc|cc}
% \R & \cdots & \R & {\R_+} & {\R_+} \\
% \R & \cdots & \R & {\R_+} & {\R_-} \\
% \midrule
% {\R_+} & {\{0\}} & {\{0\}} & {\{0\}} & {\{0\}} \\
% {\{0\}} & {\ddots} & {\{0\}} & {\vdots} & {\vdots} \\
% {\{0\}} & {\{0\}} & {\R_+} & {\{0\}} & {\{0\}} 
% \end{pNiceArray}.\text{)}
% $$
    % Before we proceed, we shall first notice that, as $\bigvee_{i,j\in[3]}\bx_i^{\T}\by_j$ only depends on the relative angles between these vectors, and $\bx_1^{\T}\by_3>0$ which implies $\bx_1\neq\bd{0}$, we may w.l.o.g.\ parameterize $\bx_1=k\cdot\be_1$, where $k\in[0,t]$ is the scaling variable. 
    To this end, suppose that there are some $\bX^{\star},\bY^{\star}\in\R^{(n-1)\times n}$ such that ${\bX^{\star}}^{\T}\bY^{\star}\in\bbO$. We first claim that $\bx^{\star}_3,\ldots,\bx^{\star}_n$ are linearly independent by showing the equation $\sum_{i=3}^n a_i\cdot \bx^{\star}_i=\bd{0}$ where $a_3,\ldots,a_n\in\R$ are variables to be solved has the only solution $a_3,\ldots,a_n=0$. To see this, we first observe from ${\bx^{\star}_{j+2}}^{\T}\by^{\star}_j>0$ for all $j\in[n-2]$ that $\bx^{\star}_3,\ldots,\bx^{\star}_n$ and $\by^{\star}_1,\ldots,\by^{\star}_{n-2}$ are all nonzero. Then, we multiply $\by^{\star}_j$ on both sides of the equation $\sum_{i=3}^n a_i\cdot \bx^{\star}_i=\bd{0}$ for each $j\in[n-2]$, leading to a system of equations $\sum_{i=3}^n a_i \cdot{\bx^{\star}_i}^{\T}\by^{\star}_j=0$, $j\in[n-2]$. Because ${\bx^{\star}_i}^{\T}\by^{\star}_j=0$ for all $i\in[3,n]\cap\N$ and $j\in[n-2]$ with $i-j\neq 2$ and is strictly positive otherwise due to the assumption, the above system of equations simplifies to $a_{j+2} \cdot{\bx^{\star}_{j+2}}^{\T}\by^{\star}_j=0$, $j\in[n-2]$, which has the only solution $a_3,\ldots,a_n=0$, as desired. We next observe from ${\bx^{\star}_i}^{\T}\by^{\star}_{n-1},{\bx^{\star}_i}^{\T}\by^{\star}_n=0$ for all $i\in[3,n]\cap\N$ that $\by^{\star}_{n-1},\by^{\star}_n\in(\operatorname{span}(\{\bx^{\star}_i\}_{i=3}^n))^\perp$, which is a one-dimensional subspace due to the above linear independence claim.
    % , i.e., a line passing through the origin.
    Therefore, because ${\bx_1^{\star}}^{\T}\by_{n-1}^{\star}>0$ and ${\bx_1^{\star}}^{\T}\by_n^{\star}>0$, we know $\by_{n-1}^{\star}$ and $\by_n^{\star}$ are also nonzero and their directions are the same. However, this would further imply ${\bx_2^{\star}}^{\T}\by_{n}^{\star}>0$ provided ${\bx_2^{\star}}^{\T}\by_{n-1}^{\star}>0$, a contradiction to the assumption. 
\end{example}

Notice that although the above example is only restricted to the case where $m=n$, in fact, it can be trivially extended to any rectangular case by canonical embedding. Therefore, the above example, together with Proposition~\ref{prop:surj-mf}, immediately implies the following interesting corollary.

\begin{corollary}\label{cor:mf-iff}
    % Let $m,n,d\in\N$. Then, the mapping 
    % $\HI_{\operatorname{MF}}(\{\bx_i\}_{i=1}^m,\{\by_j\}_{j=1}^n):=\operatorname{vec}(\bigvee_{i,j=1}^{m,n}\bx_i^{\T}\by_j)$ from $(\R^{d})^{m+n}$ to $\R^{m n}$
    The map $\HI_{\operatorname{MF}}$ (\ref{eq:mapping-MF})
    is locally surjective at the origin if and only if $d\ge\min\{m,n\}$.
\end{corollary}

We close this subsection by discussing another recent work~\cite{levin2024effect} whose results can have some implications for the local surjection property of $\HI_{\operatorname{MF}}$ (\ref{eq:mapping-MF}). (In what follows, for any $d\in\N$, we denote $\R^{m\times n}_{\le d}$ as the set of all $m\times n$ matrices with rank at most $d$). Indeed, by combining~\cite[Theorem~2.3]{levin2024effect} and~\cite[Proposition~2.8]{levin2024effect}, we can obtain an equivalent condition for the local surjection property of $\HI_{\operatorname{MF}}$ at some point $(\bX,\bY)\in\R^{d\times m}\times \R^{d\times n}$ with the codomain restricted as $\R^{m\times n}_{\le d}$ (instead of $\R^{m\times n}$) when $1\le d\le\min\{m,n\}$, namely
\begin{equation}\label{eq:levin}
\operatorname{rank}(\bX)=\operatorname{rank}(\bY)=\operatorname{rank}(\bX^{\T}\bY).
\end{equation}
We next highlight the following differences between the above condition and ours:

\begin{itemize}
    \item Because the tools developed by Clarke~\cite{clarke1990optimization} are defined and played on Banach spaces only, but $\R^{m\times n}_{\le d}$ is not even a vector space when $d\le\min\{m,n\}-1$, for our purposes, we can only apply such a result when $d=\min\{m,n\}$. Although this gives some sharper characterization of the local surjection property of $\HI_{\operatorname{MF}}$ in this case, it does not deliver any information about the situation when $d<\min\{m,n\}$. By contrast, in this paper, we have revealed a fundamental impossibility of the local surjection property of $\HI_{\operatorname{MF}}$ under such a regime even at the origin through a careful construction (cf.~Example~\ref{ex:mf-general}).
    
    \item Although we can generalize one direction of the proof of~\cite[Proposition~2.8]{levin2024effect} to see (\ref{eq:levin}) still suffices for the local surjection property of $\HI_{\operatorname{MF}}$ under the overparameterized setting\footnote{We remark that it is not possible to generalize the other direction.} (i.e., $d>\min\{m,n\}$), such a result is weaker than ours when $d\ge m+n+\min\{m,n\}$, where we have shown in Theorem~\ref{thm:surj-mf-anypoint} that $\HI_{\operatorname{MF}}$ is locally surjective everywhere in this case. Besides, during the interpolation $d\in(\min\{m,n\},m+n+\min\{m,n\})\cap\N$, it is also fairly easy to construct examples of $\bX$ and $\bY$ such that (\ref{eq:levin}) does not hold but our (\ref{eq:surj-mf-anypoint}) in Theorem~\ref{thm:surj-mf-anypoint}
    % , i.e., 
    % $$
    % d\ge\operatorname{dim}(\operatorname{span}(\{\bx_i\}_{i=1}^m\cup\{\by_j\}_{j=1}^n))+\min\{m,n\},
    % $$
    does, showing the sharpness of our results. As the simplest example, suppose that $d=\min\{m,n\}+1$, $\bX$ is an arbitrary rank-one matrix, and $\bY=\bO$. Then, we have
    $$
    \operatorname{rank}(\bX)=1\neq 0=\operatorname{rank}(\bY)=\operatorname{rank}(\bX^{\T}\bY),
    $$
    but
    $$
    \operatorname{dim}(\operatorname{span}(\{\bx_i\}_{i=1}^m\cup\{\by_j\}_{j=1}^n))+\min\{m,n\}=1+\min\{m,n\}=d,
    $$
    as desired.
\end{itemize}

% \begin{corollary}
%     Suppose $g:\R^{m\times n}\rightarrow\R$ is locally Lipschitz and $d\ge m+n+\min\{m,n\}$, i.e., the problem is overparameterized. Then, at \textnormal{any} point, it holds that for $f(\bX,\bY):=g(\operatorname{vec}(\bX^{\T}\bY))$, 
%     $$\partial_C f(\bX,\bY)=\begin{pmatrix}
%         \bA_1\by & \bA_2\by & \cdots & \bA_{mn}\by \\
%         \bA_1^{\T}\bx & \bA_2^{\T}\bx & \cdots & \bA_{mn}^{\T}\bx
%     \end{pmatrix}\cdot\partial_C g(\operatorname{vec}(\bX^{\T}\bY)),
%     $$ 
%     where $\bA_i$ is defined to be the $i$th element in $\mathbb{A}$.
% \end{corollary}

\subsection{Local surjection property, subdifferential chain rule of factorization machine}
With the local surjection property of $\HI_{\operatorname{MF}}$ (\ref{eq:mapping-MF}) well studied, we next turn our attention to the FM. To begin with, we introduce
% systematically study the local surjection property of 
the following mapping associated with the FM
\begin{equation}\label{eq:mapping-FM}
    \HI_{\operatorname{FM}}(\bP;\ba):=\bigvee_{j>i}(a_{i j}\cdot\bp_i^{\T}\bp_j),\quad\text{from $\R^{d\times d_0}$ to $\R^{(d_0-1)d_0/2}$, where $a_{i j}\neq 0$ for all $j>i$},
\end{equation}
% \begin{equation}\label{eq:mapping-FM} \begin{aligned} \HI_{\operatorname{FM}}:(\R^{d})^{d_0}&\rightarrow\R^{(d_0-1)d_0/2}\\ (\{\bp_i\}_{i=1}^{d_0})&\mapsto\bigvee_{j>i}(a_{i j}\cdot\bp_i^{\T}\bp_j), \end{aligned} \end{equation}
and we have abused our notation to denote $\ba:=\bigvee_{j>i}a_{i j}$.
% (We remark that the nonzero assumption is w.l.o.g., as otherwise we can remove the corresponding coordinates and apply Remark~\ref{rmk:special-case}).
Our first result in this subsection parallels that of Theorem~\ref{thm:surj-mf-anypoint}. 
Although the two examples (i.e., MF and FM) share some similarity, this parallel turns out to be highly nontrivial.
% Given the establishments in the previous part,
% It is tempting to believe the parallel to be relatively simple as the two examples share some similarity, but it turns out to be not the case at all. 
% and to achieve the goal, some new and tricky arguments have to be applied, as we shall soon see.

\begin{theorem}\label{thm:fm-surj1}
    % Let $d_0\in\N$ and
    % Suppose $a_{i j}\neq 0$, $\forall\,j>i$. Then, 
    Given some interested point $\bP\in\R^{d\times d_0}$, provided
    $$
    d\ge \operatorname{dim}(\operatorname{span}(\{\bp_i\}_{i=1}^{d_0}))+d_0-1,
    $$
    the mapping $\HI_{\operatorname{FM}}(\bullet;\ba)$
    % $\HI_{\operatorname{FM}}(\{\bp_i\}_{i=1}^{d_0})\allowbreak:=\allowbreak\bigvee_{j>i}(a_{i j}\cdot\bp_i^{\T}\bp_j)$ from $(\R^{d})^{d_0}$ to $\R^{(d_0-1)d_0/2}$ 
    (\ref{eq:mapping-FM})
    is locally surjective at this point. In particular, provided
    \begin{equation}\label{eq:surj-fm-everywhere}
    d\ge 2d_0 -1,
    \end{equation}
    the mapping $\HI_{\operatorname{FM}}(\bullet;\ba)$ (\ref{eq:mapping-FM}) is locally surjective everywhere.
\end{theorem}

\begin{proof}
    Similar to showing Theorem~\ref{thm:surj-mf-anypoint}, we first observe that $\forall\,t>0$, $\exists\,\varepsilon>0$, such that
    % , $\HI_{\operatorname{FM}}(\{\bp_i\}_{i=1}^{d_0}+t\bbB)\supseteq\{\bigvee_{j>i}a_{i j}\cdot(\bp_i^{\T}\bp_j+\varepsilon\bp_i^{\T}\bw_j+\varepsilon\bw_i^{\T}\bp_j+\varepsilon^2\bw_i^{\T}\bw_j):\{\bw_i\}_{i=1}^{d_0}\subseteq\bbB\}$. Therefore, by letting $\bbV:=\operatorname{span}(\{\bp_i\}_{i=1}^{d_0})$, we further have that 
    $$
    \HI_{\operatorname{FM}}(\bP+t\bbB;\ba)\supseteq\bigvee_{j>i}(a_{i j}\cdot\bp_i^{\T}\bp_j)+\varepsilon^2\cdot\left\{\bigvee_{j>i}(a_{i j}\cdot\bw_i^{\T}\bw_j):\bw_i\in\TP_{\bbV^\perp}(\bbB),\,\forall\,i\in[d_0]\right\},
    $$
    where 
    $$
    \bbV:=\operatorname{span}(\{\bp_i\}_{i=1}^{d_0}).
    $$
    Due to similar reasons, we may further assume that $\TP_{\bbV^\perp}(\bbB)$ is aligned with the standard bases. Next, we manage to apply some ``peeling'' arguments
    % (see, e.g., the proof of~\cite[Theorem~9.34]{wainwright2019high})
    to conclude the proof. To begin with, let us define the following sets
    \begin{equation}\label{eq:bbU-definition}
    \bbU_i:=\left(\prod_{j=1}^i \left[-1/(\sqrt{2})^{d_0-j},1/(\sqrt{2})^{d_0-j}\right]\right)\times\{0\}^{d-\operatorname{dim}(\bbV)-i}\subseteq\R^{d-\operatorname{dim}(\bbV)},\quad\text{for $i\in[d_0-1]$}.
    \end{equation}
    (Notice that because $d-\operatorname{dim}(\bbV)\ge d_0-1$ as assumed, the above definitions are legitimate.) Because
    $$
        \sum_{j=1}^i\left(1/(\sqrt{2})^{d_0-j}\right)^2\le\sum_{j=1}^{d_0-1}\left(1/(\sqrt{2})^{d_0-j}\right)^2<
        \sum_{j=1}^{\infty}\frac{1}{2^j}=1,\quad\text{for every $i\in[d_0-1]$},
    $$
    we know $\bbU_i\subseteq\TP_{\bbV^\perp}(\bbB)$ for all $i\in[d_0-1]$. Hence, it holds that
    $$
    \HI_{\operatorname{FM}}(\bP+t\bbB;\ba)\supseteq\bigvee_{j>i}(a_{i j}\cdot\bp_i^{\T}\bp_j)+\varepsilon^2\cdot\left\{\bigvee_{j>i}(a_{i j}\cdot\bw_i^{\T}\bw_j):\bw_i\in\bbU_i,\,\forall\,i\in[d_0-1],\,\bw_{d_0}\in\bbU_{d_0-1}
    % \{\bw_i\}_{i=1}^{d_0}\subseteq\left(\prod_{i=1}^{d_0-1}\bbU_i\right)\times\bbU_{d_0-1}
    \right\}.
    $$
    Due to the sparse support of each $\bbU_i$, it is easy to observe 
    % that each element of the last set above has 
    the following tower-like structure 
    \begin{equation}\label{eq:tower-structure}
    \bw_i\in\bbU_i,\,\forall\,i\in[d_0-1],\,\bw_{d_0}\in\bbU_{d_0-1}\quad\implies\quad\bigvee_{j>i}(a_{i j}\cdot\bw_i^{\T}\bw_j)=\bigvee_{j>i}\left(a_{i j}\sum_{k=1}^{i}w_{k i}w_{k j}\right).
    % \bigvee_{i=1}^{d_0-1}\left(\bigvee_{j=i+1}^{d_0}a_{i j}\sum_{k=1}^{i}w_{k i}w_{k j}\right).
    % ~\text{where $w_{k i}\in\left[-1/(\sqrt{2})^{d_0-k},1/(\sqrt{2})^{d_0-k}\right]$ for each $1\le k\le i\le d_0$}.
    \end{equation}
    Then, we claim that the following set inclusion holds
    \begin{equation}\label{eq:set-inclusion}
    \left\{\bigvee_{j>i}(a_{i j}\cdot\bw_i^{\T}\bw_j):\bw_i\in\bbU_i,\,\forall\,i\in[d_0-1],\,\bw_{d_0}\in\bbU_{d_0-1}\right\}\supseteq[-\rho,\rho]^{(d_0-1)d_0/2},
    \end{equation}
    which will complete the whole proof by definition, where 
    \begin{equation}\label{eq:rho-definition}
    \rho:=\frac{\min\{|a_{i j}|:j>i\}}{2^{d_0-1}}>0.
    \end{equation}
    We next prove the claim constructively. Let us consider
    $$
    \text{an arbitrary $\by\in[-\rho,\rho]^{(d_0-1)d_0/2}$, further partitioned as $\bigvee_{i=1}^{d_0-1}\by_i$, where $\by_i\in\R^{d_0-i}$, $\forall\,i\in[d_0-1]$},
    $$
    and we will show that the vectors $\bw^{\star}_1,\ldots,\bw^{\star}_{d_0}\in\R^{d-\operatorname{dim}(\bbV)}$ as recursively defined below\footnote{In the definition, all undefined elements default to zero.}
    \begin{equation}\label{eq:wstar-definition}
        w^{\star}_{i i}:=1/(\sqrt{2})^{d_0-i},\quad w^{\star}_{i j}:=\left(\be_{j-i}^{\T}\by_{i}-a_{i j}\sum_{k=1}^{i-1}w^{\star}_{k i}w^{\star}_{k j}\right)\cdot(a_{i j} w^{\star}_{i i})^{-1},~\forall\,j>i,
    \end{equation}
    satisfy the following two conditions:
    \begin{itemize}
        \item $\bigvee_{j>i}(a_{i j}\cdot{\bw^{\star}_i}^{\T}\bw^{\star}_j)=\by$,
        \item $\bw^{\star}_i\in\bbU_i,\,\forall\,i\in[d_0-1],\,\bw^{\star}_{d_0}\in\bbU_{d_0-1}$,
    \end{itemize}
    % $$
    % % \{\bw^{\star}_i\}_{i=1}^{d_0}\subseteq\left(\prod_{i=1}^{d_0-1}\bbU_i\right)\times\bbU_{d_0-1}
    % $$
    which will directly imply (\ref{eq:set-inclusion}), as desired. For the equality, it suffices to observe via (\ref{eq:tower-structure}) and (\ref{eq:wstar-definition}) that for any $i\in[d_0-1]$, the subvector of $\bigvee_{j>i}(a_{i j}\cdot{\bw^{\star}_i}^{\T}\bw^{\star}_j)$ corresponding to $\by_i$ is simply
    $$
    % \begin{aligned}
        \bigvee_{j=i+1}^{d_0}\left(a_{i j}\sum_{k=1}^{i}w^{\star}_{k i}w^{\star}_{k j}\right)=\bigvee_{j=i+1}^{d_0}(a_{i j}w^{\star}_{i i}w^{\star}_{i j})+\bigvee_{j=i+1}^{d_0}\left(a_{i j}\sum_{k=1}^{i-1}w^{\star}_{k i}w^{\star}_{k j}\right)=\bigvee_{j=i+1}^{d_0}\be_{j-i}^{\T}\by_{i}=\by_i,
    % \end{aligned}
    $$
    as desired. To prove the inclusion, we perform an induction on the rows of $\bW^{\star}\in\R^{(d-\operatorname{dim}(\bbV))\times d_0}$. (Because the zero elements of $\bw^{\star}_1,\ldots,\bw^{\star}_{d_0}$ as defined in (\ref{eq:wstar-definition}) trivially satisfy the desired conditions demanded by $\bbU_i$'s in (\ref{eq:bbU-definition}),
    % demanded by $\left(\prod_{i=1}^{d_0-1}\bbU_i\right)\times\bbU_{d_0-1}$, 
    in what follows we shall exempt them from discussions.) Specifically, for the base case (i.e., the first row), we know from (\ref{eq:wstar-definition}) that $w^{\star}_{1 1}=1/(\sqrt{2})^{d_0-1}$, which trivially satisfies the desired condition demanded by $\bbU_1$, and $w^{\star}_{1 j}=(\be_{j-1}^{\T}\by_{1})\cdot(a_{1 j} w^{\star}_{1 1})^{-1}$ for all $j\in[2,d_0]\cap\N$. To finish the proof of this case, it remains to estimate
    $$
    |w^{\star}_{1 j}|=\left|(\be_{j-1}^{\T}\by_{1})\cdot(a_{1 j} w^{\star}_{1 1})^{-1}\right|\le\frac{\rho\cdot(\sqrt{2})^{d_0-1}}{|a_{1 j}|}\le\frac{|a_{1 j}|}{2^{d_0-1}}\cdot\frac{(\sqrt{2})^{d_0-1}}{|a_{1 j}|}=\frac{1}{(\sqrt{2})^{d_0-1}},~\forall\,j\in[2,d_0]\cap\N,
    $$
    as desired, where in the first inequality we have used the assumption that $\by\in[-\rho,\rho]^{(d_0-1)d_0/2}$, and the last the definition of $\rho$ in (\ref{eq:rho-definition}). Next, suppose that the first $\ell-1$ rows of $\bW^{\star}$ all satisfy the desired conditions, where $\ell\in[2,d_0-1]\cap\N$. (Recall that only the first $d_0-1$ coordinates of $\bbU_i$'s can possibly be nonzero as per (\ref{eq:bbU-definition}), therefore it suffices to perform the induction up to $\ell=d_0-1$.) Then, for the $\ell$-th row, we first notice similarly that $w^{\star}_{\ell \ell}=1/(\sqrt{2})^{d_0-\ell}$, which also trivially satisfies the desired condition. But before we proceed to handle the other elements of $\bW^{\star}$ in this row, as a preparation, we first estimate via a simple geometric series summation that 
    $$
    \left|\sum_{k=1}^{\ell-1}w^{\star}_{k \ell}w^{\star}_{k j}\right|\le\sum_{k=1}^{\ell-1}\frac{1}{2^{d_0-k}}= \frac{1}{2^{d_0-\ell}}-\frac{1}{2^{d_0-1}},\quad\forall\,j\in[\ell+1,d_0]\cap\N,
    $$
    where in the inequality we have used the inductive hypothesis. With this estimation at hand, we can then use the triangle inequality to deduce
    $$
    \begin{aligned}
        |w^{\star}_{\ell j}|&=\left|\left(\be_{j-\ell}^{\T}\by_{\ell}-a_{\ell j}\sum_{k=1}^{\ell-1}w^{\star}_{k \ell}w^{\star}_{k j}\right)\cdot(a_{\ell j} w^{\star}_{\ell \ell})^{-1}\right|
        \\&\le\frac{\rho\cdot(\sqrt{2})^{d_0-\ell}}{|a_{\ell j}|}+(\sqrt{2})^{d_0-\ell}\cdot\left|\sum_{k=1}^{\ell-1}w^{\star}_{k \ell}w^{\star}_{k j}\right|
        \\&\le(\sqrt{2})^{d_0-\ell}\cdot\left(\frac{1}{2^{d_0-1}}+\frac{1}{2^{d_0-\ell}}-\frac{1}{2^{d_0-1}}\right)=\frac{1}{(\sqrt{2})^{d_0-\ell}},\quad\forall\,j\in[\ell+1,d_0]\cap\N,
    \end{aligned}
    $$
    as desired, which completes the proof of the induction, and thus the whole proof.
\end{proof}

By combining Theorem~\ref{thm:fm-surj1} and Remark~\ref{rmk:special-case}, the following corollary is quite straightforward.

\begin{corollary}\label{cor:fm-chain-rule}
    % Let $n\in\N$ and 
    Let $\ell:\R^{(d_0-1)d_0/2}\rightarrow\R$ be locally Lipschitz. Provided (\ref{eq:surj-fm-everywhere}) holds, the subdifferential chain rule of $\ell\circ\HI_{\operatorname{FM}}(\bullet;\ba)$, where $\HI_{\operatorname{FM}}(\bullet;\ba)$ was defined in (\ref{eq:mapping-FM}), holds everywhere.
\end{corollary}

We next apply Corollary~\ref{cor:fm-chain-rule} to the subdifferential computation in training a practical FM.

\begin{corollary}\label{cor:FM-subdiff-comp}
    Suppose that (\ref{eq:surj-fm-everywhere}) and the following qualification hold simultaneously:
    \begin{equation}\label{eq:dataset-qualification}
    \text{the dataset $\mathbb{T}:=\{(\bx_i,y_i)\}_{i=1}^n\subseteq\R^{d_0}\times\R$ satisfies $|\operatorname{supp}(\bx_i)\cap\operatorname{supp}(\bx_j)|\le 1$ for all $j>i$}.
    \end{equation}
    Then, it holds for the following function (recall from (\ref{eq:c-fm}) the definition of $f_{\operatorname{FM}}$)
    % and we manage to compute the subdifferential of the following objective function:
    $$
        L(\bP;\mathbb{T}):=\sum_{k=1}^n\ell_k(f_{\textnormal{FM}}(\boldsymbol{x}_k;\bP)),\quad\text{from $\R^{d\times d_0}$ to $\R$},
    $$
    where $\ell_k:\R\rightarrow\R$ is locally Lipschitz for every $k\in[n]$, and any $\bP\in\R^{d\times d_0}$ that
    $$
    \partial_C L(\bP;\mathbb{T})=\sum_{k=1}^n\partial_C\ell_k(f_{\textnormal{FM}}(\boldsymbol{x}_k;\bP))\cdot\nabla_{\bP}f_{\textnormal{FM}}(\boldsymbol{x}_k;\bP).
    $$
\end{corollary}

It is worth noting that the qualification (\ref{eq:dataset-qualification}) is always satisfied under an ordinary recommender system scenario; cf.~\cite[Section~4.1.1]{rendle2012factorization}.

\begin{proof}
    Because of (\ref{eq:dataset-qualification}), we can compositely rewrite $L(\bullet;\mathbb{T})=\ell(\bullet;\mathbb{T})\circ\HI^{\mathbb{D}_{\mathbb{T}}}_{\operatorname{FM}}(\bullet;\ba_{\mathbb{T}})$,
    % which is amenable to Corollary~\ref{cor:fm-chain-rule},
    where
    % $\ba$ and $\mathbb{D}_{\mathbb{T}}$ are such that:
    $$
       \ell(\by;\mathbb{T}):=\sum_{k=1}^n\ell_k\left(\sum_{(i,j)\in\mathbb{D}_{\mathbb{T},k}}y_{\mathcal{I}(i,j)}\right),\quad\text{from $\R^{|\mathbb{D}_{\mathbb{T}}|}$ to $\R$},
    $$
    the mapping $\mathcal{I}:\mathbb{D}_{\mathbb{T}}\rightarrow[|\mathbb{D}_{\mathbb{T}}|]$ is the lexicographical bijection between $\mathbb{D}_{\mathbb{T}}$ and $[|\mathbb{D}_{\mathbb{T}}|]$,
    $$
        \mathbb{D}_{\mathbb{T}}:=\bigcup_{k=1}^n\mathbb{D}_{\mathbb{T},k},~\ba_{\mathbb{T}}:=\bigvee_{j>i}\left(\bigcup_{k=1}^n\{x_{i k} x_{j k}:(i,j)\in\mathbb{D}_{\mathbb{T},k}\}\right),~\mathbb{D}_{\mathbb{T},k}:=\{(i,j):i,j\in\operatorname{supp}(\bx_k),\,j>i\},
    $$
    % that is:
    % $$
    % \HI^{\mathbb{D}_{\mathbb{T}}}_{\operatorname{FM}}(\bullet;\ba)=\bigvee_{i\in[n],\,\operatorname{supp}(\bx_i)\ge 2} f_{\textnormal{FM}}(\boldsymbol{x}_i;\bullet),\quad\text{up to a global permutation},
    % $$
    and $\HI^{\mathbb{D}_{\mathbb{T}}}_{\operatorname{FM}}(\bullet,\ba_{\mathbb{T}})$ is defined similarly as in Remark~\ref{rmk:special-case}. (In plain words, the construction of $\mathbb{D}_{\mathbb{T}}$ is such that all the feature interactions appeared in $f_{\textnormal{FM}}(\boldsymbol{x}_1;\bullet),\ldots,f_{\textnormal{FM}}(\boldsymbol{x}_n;\bullet)$ are exactly contained in it.) Because of (\ref{eq:dataset-qualification}) again, we know $\mathbb{D}_{\mathbb{T},k}$'s are mutually disjoint, and we can thus observe a separable structure in $\ell(\bullet;\mathbb{T})$.
    % , makes its subdifferential easily computable provided so is each $\ell_k$. (As $\ell_k$'s are univariate, this should be the case.)
    This, together with~\cite[Proposition~2.5]{rockafellar1985extensions},~\cite[Theorem~2.3.10]{clarke1990optimization}, Remark~\ref{rmk:special-case}, and Corollary~\ref{cor:fm-chain-rule}, directly implies the desired result.
    % the easy computability of 
    % where no contextual information beyond the ratings is given
\end{proof}

Similar to the previous part, with the sufficient condition (Theorem~\ref{thm:fm-surj1}) established, it is very natural and interesting to ask if it is also necessary (at least at the origin). With the ideas used in Example~\ref{ex:mf-general}, this time it becomes not that hard to construct the following negative example.

\begin{example}\label{ex:fm-general}
    Consider the case where $d_0\ge 3$ is arbitrary and $d=d_0-2$.
    % mapping $\HI_{\operatorname{FM}}(\bP):=\bigvee_{j>i}(a_{i j}\cdot\bp_i^{\T}\bp_j)$ from $(\R^{d_0-2})^{d_0}$ to $\R^{(d_0-1)d_0/2}$ where $a_{i j}\neq 0$ for all $1\le i<j\le d_0$ and $d_0\in\N$. 
    We manage to refute the local surjection property of $\HI_{\operatorname{FM}}(\bullet;\ba)$ (\ref{eq:mapping-FM}) at the origin by showing
    % for any $t>0$, 
    $$
    \HI_{\operatorname{FM}}(\R^{(d_0-2)\times d_0};\ba)\cap\bbO(\ba)=\varnothing,
    $$
    where
    $$
        \bbO(\ba):=\operatorname{relint}\left(\R_+^{d_0-3}\times\left(\prod_{j=d_0-1}^{d_0}(\operatorname{sign}(a_{1 j})\cdot\R_+)\right)\times\{0\}^{(d_0-2)(d_0-1)/2}
        % \times(\operatorname{sign}(a_{d_0-1,d_0})\cdot\R_-)
        \right).
    $$
    \begin{wrapstuff}[r,width=.29\textwidth,abovesep=-5pt]
    \begin{equation*}
\begin{pNiceArray}{ccccc}
\multicolumn{1}{c|}{\phantom{\R_+}} & {\R_+} & {\R_+} & \cdots & {\R_+} \\ \cline{2-2}
\phantom{\R_+} & \multicolumn{1}{c|}{\phantom{\R_+}} & {\{0\}} & \cdots & {\{0\}} \\ \cline{3-3}
\phantom{\R_+} & \phantom{\R_+} & \multicolumn{1}{c|}{\phantom{\R_+}} & \ddots & \vdots \\ \cline{4-4}
\phantom{\R_+} & \phantom{\R_+} & \phantom{\R_+} & \multicolumn{1}{c|}{\phantom{\R_+}} & {\{0\}} \\ \cline{5-5}
\phantom{\R_+} & \phantom{\R_+} & \phantom{\R_+} & \phantom{\R_+} & \phantom{\R_+}
 % \\ \cline{4-4}
\end{pNiceArray}
\end{equation*}
\end{wrapstuff}
\noindent (A visualization of the folded $\bbO(\ba)$ with $a_{1 j}>0$ for $j=d_0-1,d_0$ before taking relative interior is given in the right panel, where the ordinate and abscissa refer to $i$ and $j$ respectively.) To this end, suppose that there is some $\bP^{\star}\in\R^{(d_0-2)\times d_0}$ with $\bigvee_{j>i}(a_{i j}\cdot{\bp_i^{\star}}^{\T}\bp_j^{\star})\in\bbO(\ba)$. From ${\bp_1^{\star}}^{\T}\bp_i^{\star}\neq 0$ for all $i\in[2,d_0-2]\cap\N$, we know $\bp_i^{\star}\neq\bd{0}$ for each $i\in[d_0-2]$. Besides, from ${\bp_i^{\star}}^{\T}\bp_j^{\star}=0$ for all $i,j\in[2,d_0-2]\cap\N$ with $i\neq j$, we know $\{\bp_i^{\star}\}_{i=2}^{d_0-2}$ forms an orthogonal basis for $\operatorname{span}(\{\bp_i^{\star}\}_{i=2}^{d_0-2})$, which is therefore a $(d_0-3)$-dimensional subspace of $\R^{d_0-2}$. As ${\bp_i^{\star}}^{\T}\bp_{d_0-1}^{\star},{\bp_i^{\star}}^{\T}\bp_{d_0}^{\star}=0$ for all $i\in[2,d_0-2]\cap\N$, it follows that $\bp_{d_0-1}^{\star},\bp_{d_0}^{\star}\in(\operatorname{span}(\{\bp_i^{\star}\}_{i=2}^{d_0-2}))^\perp$, a one-dimensional subspace. Therefore, as ${\bp_1^{\star}}^{\T}\bp_{d_0-1}^{\star},{\bp_1^{\star}}^{\T}\bp_{d_0}^{\star}>0$, we know $\bp_{d_0-1}^{\star}$ and $\bp_{d_0}^{\star}$ are also nonzero with the same direction,
% , and $\bp_{d_0-1}^{\star}=\beta\bp_{d_0}^{\star}$ for some $\beta>0$,
% are also nonzero with the same direction, 
which further implies $\langle\bp_{d_0-1}^{\star},\bp_{d_0}^{\star}\rangle>0$, a contradiction.
\end{example}

Summarizing Theorem~\ref{thm:fm-surj1} and Example~\ref{ex:fm-general} leads to the following interesting corollary.

\begin{corollary}\label{cor:fm-iff}
    % Let $d_0\in\N$ and 
    % Suppose $a_{i j}\neq 0$, $\forall\,j>i$. Then, t
    The map $\HI_{\operatorname{FM}}(\bullet;\ba)$ 
    (\ref{eq:mapping-FM})
    % $\HI_{\operatorname{FM}}(\{\bp_i\}_{i=1}^{d_0}):=\bigvee_{j>i}(a_{i j}\cdot\bp_i^{\T}\bp_j)$ from $(\R^{d})^{d_0}$ to $\R^{(d_0-1)d_0/2}$ 
    is locally surjective at the origin if and only if $d\ge d_0-1$.
\end{corollary}

\subsection{Tensor generalizations}\label{sec:lsp-cp}
In this subsection, we manage to extend the results in Section~\ref{sec:lsp-mf} to higher-order cases, i.e., study the local surjection property of the following mapping
\begin{equation}\label{eq:mapping-CP}
\HI_{\operatorname{CP}}(\bX,\bY,\bZ):=\llbracket\bX,\bY,\bZ\rrbracket,\quad\text{from $\R^{d\times n_1}\times\R^{d\times n_2}\times\R^{d\times n_3}$ to $\R^{n_1\times n_2\times n_3}$},
\end{equation}
corresponding to the tensor CP factorization example. As we shall soon see in the next part, despite its classicality, the above mapping (as well as its close relatives) actually has a close connection to a variety of modern neural networks. Unfortunately, regarding this similar but generalized example, we only have some very partial results. The first one parallels that of Proposition~\ref{prop:surj-mf}.

\begin{proposition}\label{prop:surj-CP-origin}
    % Let $n_1,n_2,n_3,d\in\N$. Then, 
    Provided $d\ge\min\{\prod_{j\ne i}n_j\}_{i=1}^3$, the map $\HI_{\operatorname{CP}}$ 
    % (\ref{eq:mapping-CP})
    % $\HI(\{\bx_i\}_{i=1}^{n_1},\allowbreak\{\by_j\}_{j=1}^{n_2},\{\bz_k\}_{k=1}^{n_3}):=\operatorname{vec}(\bigvee_{i,j,k=1}^{n_1,n_2,n_3}\langle\bx_i,\by_j,\bz_k\rangle)$ from $(\R^d)^{n_1+n_2+n_3}$ to $\R^{n_1 n_2 n_3}$
    is locally surjective at the origin.
\end{proposition}

\begin{proof}
    Suppose w.l.o.g.\ that $\min\{\prod_{j\ne i}n_j\}_{i=1}^3=n_1 n_3$. We first observe similarly that for any $t>0$, there exists some $\varepsilon>0$, such that $\HI_{\operatorname{CP}}((\bX,\bY,\bZ)+t\bbB)$ contains
    $$
    \llbracket\bX,\bY,\bZ\rrbracket+\varepsilon^3\cdot\left\{\llbracket\bA,\bB,\bC\rrbracket:\ba_i\in\bbB,\,\forall\,i\in[n_1],\,\bb_j\in\bbB,\,\forall\,j\in[n_2],\,\bc_k\in\bbB,\,\forall\,k\in[n_3]
    % \{\}_{i=1}^{n_1},\{\bb_j\}_{j=1}^{n_2},\{\bc_k\}_{k=1}^{n_3}\subseteq\bbB
    \right\}.
    $$
    Then, by letting
    $$
    \ba_i:=(\be_i^{n_1}\boxtimes\bd{1}_{n_3}/\sqrt{n_3})\vee\bd{0}_{d-n_1 n_3}\in\bbB,\quad\text{for each $i\in[n_1]$},
    $$
    which is legitimate as $d\ge n_1 n_3$, and partitioning $\bb_j=\bigvee_{i=1}^{n_1+1}\bb_{j,i}$ and $\bc_k=\bigvee_{i=1}^{n_1+1}\bc_{k,i}$ accordingly, where $\bb_{j,i},\bc_{k,i}\in\R^{n_3}$ for all $i\in[n_1]$, we see $\HI_{\operatorname{CP}}((\bX,\bY,\bZ)+t\bbB)$ further contains
    $$
    \llbracket\bX,\bY,\bZ\rrbracket+\varepsilon^3\cdot\left\{\G(\bB,\bC):\bb_{j,i}\in\bbB/\sqrt{n_1},\,\forall\,(i,j)\in[n_1]\times [n_2],\,\bc_{k,i}\in\bbB/\sqrt{n_1},\,\forall\,(i,k)\in[n_1]\times[n_3]\right\},
    $$
    where $\G:\R^{d\times n_2}\times\R^{d\times n_3}\rightarrow\R^{n_1\times n_2\times n_3}$ with
    $$
        (\G(\bB,\bC))_{i j k}:=\bb_{j,i}^{\T}\bc_{k,i}/\sqrt{n_3},\quad\text{for all $(i,j,k)\in[n_1]\times[n_2]\times[n_3]$}.
    $$
    Because the mode-$1$ slices of $\G(\bB,\bC)$ are independent of each other and the lengths of $\bb_{j,i}$ and $\bc_{k,i}$ satisfy $n_3\ge\min\{n_2,n_3\}$, we can simply apply
    % , the desired claim thus directly follows from the same arguments used in the proof of 
    Proposition~\ref{prop:surj-mf} to each model-$1$ slice of $\G(\bB,\bC)$, and then invoke (\ref{eq:easy-fact}) to conclude the proof.
\end{proof}

The second result considers ``pseudo tensors'' that are essentially matrices.

\begin{proposition}\label{prop:pseudo-tensors}
    % Let $n_2,n_3,d\in\N$. Then, 
    If $d\ge n_2+n_3+\min\{n_2, n_3\}$,
    % $$
    % d\ge\operatorname{dim}\left(\operatorname{span}\left(\left\{\sqrt{|\bx|}\oast\by_j\right\}_{j=1}^{n_2}\bigcup\left\{\left(\operatorname{sign}(\bx)\oast\sqrt{|\bx|}\right)\oast\bz_k\right\}_{k=1}^{n_3}\right)\right)+\min\left\{\prod_{j\ne i}n_j\right\}_{i=1}^3,
    % $$
    then the below map is locally surjective everywhere 
    $$
    \HI^{\dagger}_{\operatorname{CP}}(\bx,\bY,\bZ):=\operatorname{squeeze}(\llbracket\bx,\bY,\bZ\rrbracket)=\bY^{\T}\operatorname{diag}(\bx)\bZ,\quad\text{from $\R^{d}\times\R^{d\times n_2}\times\R^{d\times n_3}$ to $\R^{n_2\times n_3}$}.
    $$
\end{proposition}

\begin{proof}
    To begin with, we suppose w.l.o.g.\ that $x_i\neq 0$ for each $i\in[d]$. (For the remaining case, it suffices to slightly perturb $\bx$ along the direction of $\sum_{i\notin\operatorname{supp}(\bx)}\be_i$ to turn it into ``general position'';
    % , where $\bbI:=\{i\in[d]:x_i=0\}$. 
    this is feasible since we are always working within some neighborhoods.) Then, we observe as before that for any $t>0$, there exists some $\varepsilon>0$, such that $\HI^{\dagger}_{\operatorname{CP}}((\bx,\bY,\bZ)+t\bbB)-\bY^{\T}\operatorname{diag}(\bx)\bZ$ contains 
    % (notice that here we do not perturb $\bx$ anymore):
    $$
    \left\{\varepsilon\cdot\bB^{\T}\operatorname{diag}(\bx)\bZ+\varepsilon\cdot\bY^{\T}\operatorname{diag}(\bx)\bC+\varepsilon^2\cdot\bB^{\T}\operatorname{diag}(\bx)\bC:\bb_j\in\bbB,\,\forall\,j\in[n_2],\,\bc_k\in\bbB,\,\forall\,k\in[n_3]\right\}.
    $$ 
    Next, let us factorize $\bx=(\operatorname{sign}(\bx)\oast\sqrt{|\bx|})\oast\sqrt{|\bx|}$, and define accordingly
    $$
    \bb_j^\prime:=\sqrt{|\bx|}\oast\bb_j,~\bc_k^\prime:=\left(\operatorname{sign}(\bx)\oast\sqrt{|\bx|}\right)\oast\bc_k,~\by_j^\prime:=\sqrt{|\bx|}\oast\by_j,~\bz_k^\prime:=\left(\operatorname{sign}(\bx)\oast\sqrt{|\bx|}\right)\oast\bz_k.
    $$
    Before we proceed, we first observe
    $$
    \sqrt{|\bx|}\oast\bbB,\left(\operatorname{sign}(\bx)\oast\sqrt{|\bx|}\right)\oast\bbB\supseteq\left(\min_{i\in[d]}\sqrt{|x_i|}\right)\cdot\bbB\neq\{\bd{0}\},
    $$
    where the second inequality follows from the the assumption made at the very beginning. As a result, we know $\HI^{\dagger}_{\operatorname{CP}}((\bx,\bY,\bZ)+t\bbB)-\bY^{\T}\operatorname{diag}(\bx)\bZ$ further contains 
    $$
    \left\{\varepsilon\cdot{\bB^{\prime}}^{\T}\bZ^{\prime}+\varepsilon\cdot{\bY^{\prime}}^{\T}\bC^{\prime}+\varepsilon^2\cdot{\bB^{\prime}}^{\T}\bC^{\prime}:\bb_j^\prime\in\kappa\cdot\bbB,\,\forall\,j\in[n_2],\,\bc_k^\prime\in\kappa\cdot\bbB,\,\forall\,k\in[n_3]\right\},
    $$
    where $\kappa:=\min_{i\in[d]}\sqrt{|x_i|}>0$. A closer inspection reveals that the situation here is almost identical to that of Theorem~\ref{thm:surj-mf-anypoint}, and we can thus repeat the arguments there to conclude the proof.
\end{proof}

After presenting the proofs, we would like to remark that, as the readers may have already sensed, the biggest difficulty in studying the local surjection property 
% of $\HI_{\operatorname{MF}}$ to that of
of $\HI_{\operatorname{CP}}$ in its full generality stems from the coordinate dependence of $\langle\cdot,\cdot,\cdot\rangle$.
% , which makes the discussions very messy. 
Unfortunately, we have no idea about how to tackle this difficulty,
% we currently have no idea about under what conditions does such a mapping possess the desired property, 
but we still believe the existence of some dimension threshold beyond which $\HI_{\operatorname{CP}}$ is locally surjective everywhere, as in the matrix case.
% $\HI(\{\bx_i\}_{i=1}^{n_1},\{\by_j\}_{j=1}^{n_2},\{\bz_k\}_{k=1}^{n_3}):=\operatorname{vec}(\bigvee_{i,j,k=1}^{n_1,n_2,n_3}\langle\bx_i,\by_j,\bz_k\rangle)$
% always holds everywhere.
% Unfortunately, we are currently not able to even provide an estimation of the threshold. 
As an aside, it is worth noting that similar difficulties also appear in some other generalizations, such as the higher-order FM~\cite{blondel2016higher} and neural FM~\cite{he2017neuralfm}, for which the interested mappings are respectively
$$
\HI_{\operatorname{HOFM}}(\bP;\ba):=\bigvee_{k>j>i}(a_{i j k}\cdot\langle\bp_i,\bp_j,\bp_k\rangle),\quad\text{from $\R^{d\times d_0}$ to $\R^{(d_0-2)(d_0-1)d_0/6}$},
$$
where $a_{i j k}\neq 0$ for all $k>j>i$, and
\begin{equation}\label{eq:neuralFM}
\HI_{\operatorname{NeuFM}}(\bP,\bH;\ba):=\bigvee_{s=1}^h\left(\bigvee_{j>i}(a_{i j}\cdot\langle\bh_s,\bp_i,\bp_j\rangle)\right),\quad\text{from $\R^{d\times d_0}\times \R^{d\times h}$ to $\R^{h(d_0-1)d_0/2}$},
\end{equation}
where $a_{i j}\neq 0$ for all $j>i$ and $h\in\N$.
% is the width of the hidden layer.
% These are left for future work.
% would still like to make the following conjecture:
% but we firmly believe that the continuation of our study in Section~\ref{sec:lsp-cp} shall produce a profound influence on the subdifferential chain rule of these related models and even beyond.

% \begin{conjecture}
%     Let $n_1,n_2,n_3,d\in\N$. Then, there exists some universal constant $C\in\R$, such that provided $d\ge m+n+\min\{n_1 n_2, n_1 n_3, n_2 n_3\}$, the above mapping has the local surjection property at any point.
% \end{conjecture}

\subsection{Neural extensions}\label{sec:neural-extensions}
In this part, we manage to apply our previous establishments to some modern neural networks. 
\subsubsection{Preparations}
To begin with, we recall the following fundamental fact on the subdifferential equivalence between the objective function for training a two-layer neural network and its partial linearization. This is a direct consequence of~\cite[Theorem~6.5]{mordukhovich1996nonsmooth} and~\cite[Theorem~8.49]{rockafellar2009variational} (see also~\cite[Theorem~46]{tian2023testing}).

\begin{fact}\label{fact:equiv}
    Let $n,h,d\in\N$. 
    Suppose for all $i\in[n]$ and $k\in[h]$ that:
    \begin{itemize}
        \item $\ell_{i}:\R\rightarrow\R$ and $g_{i k}:\R^d\rightarrow\R$ are strictly differentiable,
        \item $\sigma_{i k}:\R\rightarrow\R$ is locally Lipschitz.
        % and $g_{i k}:\R^d\rightarrow\R$ are locally Lipschitz.
        % , and $g_{i k}$ is in addition differentiable.
    \end{itemize}
    Then, at any interested point $(\bv^\prime,\bx^\prime)\in\R^h\times\R^d$, the following two functions
    % (whose arguments are clear from the contexts)
    $$
    % L(\bv,\{\bx_j\}_{j=1}^D):=
    L(\bv,\bx):=\sum_{i=1}^n\ell_{i}\left(\sum_{k=1}^{h} v_{k}\cdot\sigma_{i k}(g_{i k}(\bx))\right),\quad\overline{L}(\bv,\bx):=\sum_{k=1}^{h} v_{k}\cdot\sum_{i=1}^n p_{i}\cdot\sigma_{i k}(g_{i k}(\bx)),
    % \quad\text{from $\R^h\times\R^d$ to $\R$},
    $$
    % and 
    % $$
    % % ~\text{and}~
    % % \overline{L}(\bv,\{\bx_j\}_{j=1}^D):=
    % \overline{L}(\bv,\bx):=\sum_{k=1}^{h} v_{k}\cdot\sum_{i=1}^np_{i}\cdot\sigma_{i k}(g_{i k}(\bx)),~\text{where $p_{i}:=\ell_{i}^\prime\left(\sum_{k=1}^{h} v_{k}^\prime\cdot\sigma_{i k}(g_{i k}(\bx^{\prime}))\right)$},
    % $$
    % where 
    % $$
    % p_{i}:=\ell_{i}^\prime\left(\sum_{k=1}^{h} v_{k}^\prime\cdot\sigma_{i k}(g_{i k}(\bx^{\prime}))\right),
    % $$
    where $p_{i}:=\ell_{i}^\prime\pig(\sum_{k=1}^{h} v_{k}^\prime\cdot\sigma_{i k}(g_{i k}(\bx^{\prime}))\pig)$, share the same Clarke subdifferential.
\end{fact}

Equipped with Fact~\ref{fact:equiv}, to compute $\partial_C L$, it suffices to study the subdifferential chain rule of $\overline{L}$. Next, we give a sufficient condition for the latter based on local surjection property.
% we have the following master corollary guiding our applications:
% For convenience, we use $\overline{L}(\bv,\{\bx_j\}_{j=1}^D)$ in the sequel to denote the function of partial linearization. 
% The following lemma gives a sufficient condition for the exact subdifferential chain rule of $\overline{L}(\bv,\{\bx_j\}_{j=1}^D)$.

\begin{lemma}\label{lma:meta}
    Under the assumptions of Fact~\ref{fact:equiv}, given an arbitrary $(\bv^\prime,\bx^\prime)\in\R^h\times\R^d$, and let
    \begin{equation}\label{eq:function-h}
    h(\bv,\bX):=\sum_{k=1}^{h} v_{k}\cdot\sum_{i=1}^n p_{i}\cdot\sigma_{i k}(x_{i k}),\quad\text{from $\R^h\times\R^{n\times h}$ to $\R$}.
    \end{equation}
    Provided $\G$ is locally surjective at $\bx^{\prime}\in\R^d$, where $\G:\R^d\rightarrow\R^{n\times h}$ with
    \begin{equation}\label{eq:mapping-G}
        (\G(\bx))_{i k}:=g_{i k}(\bx),\quad\text{for all $(i,k)\in[n]\times[h]$},
    \end{equation}
    the subdifferential chain rule of $h\circ((\bv,\bx)\mapsto(\bv,\G(\bx)))=\overline{L}$
    % \begin{equation}\label{eq:Lbar-composite}
    % h\circ((\bv,\bx)\mapsto(\bv,\G(\bx)))=\overline{L}
    % \end{equation}
    holds at this point $(\bv^\prime,\bx^\prime)$.
    % $(\bv^\prime,\bx^{\prime})$
    % , $\forall\,\bv^\prime\in\R^h$. 
\end{lemma}

\begin{proof}
    % $$
    % h(\bw,\{y_{k i j}\}_{k\in[H],(i,j)\in\mathbb{D}}):=\sum_{k=1}^{H} w_{k}\cdot\sum_{(i,j)\in\mathbb{D}}p_{i j}\cdot\sigma_{i j}(y_{k i j}),~\G(\{\bx_j\}_{j=1}^D):=\bigvee_{k\in[H],(i,j)\in\mathbb{D}}\G_{k i j}(\{\bx_j\}_{j=1}^D).
    % $$
    % Then, we clearly have that $\overline{L}(\bv,\{\bx_j\}_{j=1}^D)=h(\bv,\G(\{\bx_j\}_{j=1}^D))$. 
    Since $\G$ is locally surjective at $\bx^{\prime}$, we know from (\ref{eq:easy-fact}) that the mapping $(\bv,\bx)\mapsto(\bv,\G(\bx))$ is locally surjective at $(\bv^\prime,\bx^{\prime})$,
    % where $\bv^\prime\in\R^H$ can be arbitrary
    % , due to 
    % (again) the fact that for any $k\in\N$, $\{\bz_i\}_{i=1}^k\subseteq\R^n$ and $r>0$, it always holds that $\bigvee_{i=1}^k\bz_i+r\bbB\subseteq\prod_{i=1}^k(\bz_i+r\bbB)$, as well as 
    % the separability of the arguments and (\ref{eq:easy-fact}),
    % the fact mentioned at the end of Proposition~\ref{prop:surj-mf}, 
    and the desired result thus follows from~\cite[Theorem~2.3.10]{clarke1990optimization}. 
    % Next, we manage to apply~\cite[Theorem~2.3.10]{clarke1990optimization} again to complete the proof. 
\end{proof}

% Although 
% in the above corollary 
% the subdifferential chain rule of $\overline{L}$ has been shown to hold for the aforementioned specific composition, 
With the subdifferential chain rule of $\overline{L}$ established, to finish the computation of $\partial_C \overline{L}$ (and thus $\partial_C L$), it remains to evaluate $\partial_C h$.
% to obtain the concrete expression of the subdifferential of $\overline{L}$, 
Although this may seem hard at a first sight,
% This is dealt with in the following remark:
actually a very simple characterization can be found by exploiting the separable structure of $h$.
% As a result, we can reduce the expression of $\partial_C \overline{L}$ to its simplest form. 
\begin{lemma}
\label{lma:pCh}
    Under the assumptions of Fact~\ref{fact:equiv}, it holds for any $(\bv^\prime,\bx^\prime)\in\R^h\times\R^d$ that
    $$
    % \begin{aligned}
        \partial_C h(v_1,\bx_1,\ldots,v_h,\bx_h)
        % =\prod_{k=1}^h\Big(\{g(\bx_k)\}\times\partial_C(v_k\cdot g)(\bx_k)\Big)
        =\prod_{k=1}^h\left(\left\{\sum_{i=1}^n p_{i}\cdot\sigma_{i k}(x_{i k})\right\}\times\prod_{i=1}^n(v_k p_{i}\cdot\partial_C\sigma_{i k}(x_{i k}))\right),
    % \end{aligned}
    $$
    where the function $h$ was defined in (\ref{eq:function-h}).
\end{lemma}

\begin{proof}
% [Characterization of $\partial_C h(\{w_k,\by_k\}_{k\in[H]})$]
% [Characterization of $\partial_C h$]
    % To prove the remaining, we first compute the subdifferential of $h$. 
    % By exploiting the structure, we can actually exactly compute the subdifferential of $h$. To this end, we define
    % By letting $\by_k:=\bigvee_{i=1}^n y_{k i}$ for each $k\in[h]$, we can rewrite:
    We first rewrite $h$ in a separable way
    $$
    h(\bv,\bX)=\sum_{k=1}^{h} v_{k}\cdot g_k(\bx_k),\quad\text{where $g_k(\by):=\sum_{i=1}^n p_{i}\cdot\sigma_{i k}(y_{i})$},\quad\text{from $\R^{n}$ to $\R$},
    $$
    % $$
    % h(\bw,\{y_{k i j}\}_{k\in[H],(i,j)\in\mathbb{D}})=\sum_{k=1}^{H} w_{k}\cdot g(\by_k),~\text{where}~g(\by_k):=\sum_{(i,j)\in\mathbb{D}}p_{i j}\cdot\sigma_{i j}(y_{k i j}),
    % $$
    % As we can see, both functions above admit a clear separable structure over the arguments. 
    and $g_k$ is also separable. Thus, by~\cite[Proposition~2.3.1]{clarke1990optimization} and~\cite[Proposition~2.5]{rockafellar1985extensions}, we have
    $$
    \partial_C g_k(\by)=\prod_{i=1}^n(p_{i}\cdot\partial_C\sigma_{i k}(y_i)),
    $$
    which, together with~\cite[Corollary~48]{tian2023testing} and the above references again, directly implies the claim.
    % for the following argument permutation of $h$ 
    % that:
    % where we have permuted the arguments of $h$.
    % and the desired result thus follows by~\cite[Theorem~2.3.10]{clarke1990optimization}.
    % we have that
    % $$
    % \partial_C\overline{L}(\bv^\prime,\{\bx_i^\prime\}_{i=1}^D)=
    % $$
\end{proof}

By combining Fact~\ref{fact:equiv}, Lemma~\ref{lma:meta}, and Lemma~\ref{lma:pCh}, we can reduce all the difficulties in computing $\partial_C L$ to
% we can reduce the whole story to the study of
just a matter of local surjection property.

\begin{corollary}\label{cor:compute-pCL}
    Under the assumptions of Fact~\ref{fact:equiv}, provided the mapping $\G$ (\ref{eq:mapping-G}) is locally surjective at some point $\bx^{\prime}\in\R^d$, it holds for all $\bv^{\prime}\in\R^h$ that
    $$
        \partial_C L(\bv^{\prime},\bx^{\prime})=\sum_{i=1}^n p_{i}\left(\prod_{k=1}^h\{\sigma_{i k}(g_{i k}(\bx^{\prime}))\}\times\left(\sum_{k=1}^{h} v_k^{\prime}\cdot\partial_C\sigma_{i k}(g_{i k}(\bx^{\prime}))\cdot\nabla g_{i k}(\bx^{\prime})\right)\right).
        % \left(\prod_{k=1}^h\left\{\sum_{i=1}^n p_{i}\cdot\sigma_{i k}(g_{i k}(\bx))\right\}\right)\times\left(\sum_{i=1}^n \sum_{k=1}^{h} v_k p_{i}\cdot\partial_C\sigma_{i k}(g_{i k}(\bx))\cdot\nabla g_{i k}(\bx)\right).
    $$
\end{corollary}

With Corollary~\ref{cor:compute-pCL} established, we next turn to some concrete examples.

\subsubsection{A positive result on generalized matrix factorization}
Generalized MF~\cite[Section~3.2]{he2017neural} is a vital component of the well-known neural MF method for recommender systems~\cite[Section~3.4]{he2017neural}. Specifically,
under the assumptions of Fact~\ref{fact:equiv}, 
the objective function for training a two-layer generalized MF can be formulated as\footnote{Here, we have strengthened the architecture of the generalized MF to align with other neural networks; our theory can be degraded in a straightforward manner if needed.}
$$
L_{\operatorname{GMF}}(v,\bh,\bP,\bQ;\mathbb{D}):=\sum_{(i,j)\in\mathbb{D}}\ell_{i j}(v\cdot\sigma(\langle\bh,\bp_i,\bq_j\rangle)),\quad\text{from $\R\times\R^{d}\times\R^{d\times m}\times\R^{d\times n}$ to $\R$},
$$
where $\mathbb{D}\subseteq[m]\times[n]$ for some $m,n\in\N$.
% , $\bbU:=\{i:(i,j)\in\mathbb{D}\}$, $\bbI:=\{j:(i,j)\in\mathbb{D}\}$, $m:=|\bbU|$ and $n:=|\bbI|$.
(In the context of recommender systems, $\mathbb{D}$ refers to the user-item transaction history, while $m$ and $n$ denote the number of users and items, respectively.) By specializing Corollary~\ref{cor:compute-pCL} to this setting, the mapping $\G$ (\ref{eq:mapping-G}) simply reduces to $\HI^{\dagger}_{\operatorname{CP}}$ studied in Proposition~\ref{prop:pseudo-tensors} already, and as a result, we can easily get the following corollary.

\begin{corollary}\label{cor:GMF}
    % Under the above assumptions and definitions, 
    Provided (\ref{eq:surj-mf-everywhere}) holds, it holds for any $(v,\bh,\bP,\bQ)\in\R\times\R^{d}\times\R^{d\times m}\times\R^{d\times n}$ that
    $$
        \partial_C L_{\operatorname{GMF}}(v,\bh,\bP,\bQ;\mathbb{D})=\sum_{(i,j)\in\mathbb{D}} p_{i j}\cdot\left(\{\sigma(\langle\bh,\bp_i,\bq_j\rangle)\}\times\left(v\cdot\partial_C \sigma(\langle\bh,\bp_i,\bq_j\rangle)\cdot\nabla f_{i j}(\bh,\bP,\bQ)\right)\right),
    $$
    where $p_{i j}:=\ell_{i j}^{\prime}(v\cdot\sigma(\langle\bh,\bp_i,\bq_j\rangle))$ and $f_{i j}(\bh,\bP,\bQ):=\langle\bh,\bp_i,\bq_j\rangle$ for all $(i,j)\in\mathbb{D}$.
    % the subdifferential chain rule of $L_{\operatorname{GMF}}$
    % the function $\overline{L}$ as composed in Lemma~\ref{lma:meta}
    % holds everywhere.
\end{corollary}

\subsubsection{A negative result on neural matrix factorization}
As a natural advance, one may expect to generalize the above result to the overall neural MF, whose training objective function under the assumptions of Fact~\ref{fact:equiv} can be expressed as
$$
    L_{\operatorname{NeuMF}}(\bh,\bP,\bQ,\bv,\bW,\bX,\bS,\bY,\bb;\mathbb{D}):=
    \sum_{(i,j)\in\mathbb{D}}\ell_{i j}\left(\langle\bh,\bp_i,\bq_j\rangle + \sum_{k=1}^h v_{k}\cdot\sigma(\bw_k^{\T}\bx_i+\bs_k^{\T}\by_j+b_k)\right).
$$
% where the notations and definitions are all inherited from the above, and the dimensions of the latent representations should be clear from the context. 
Unfortunately, we will show that the generalization is in general not possible, which shows the necessity to develop new techniques to compute the subdifferentials of this type of problems. 

To begin with, we use Fact~\ref{fact:equiv} to linearize $L_{\operatorname{NeuMF}}$, and then~\cite[Corollary~2.3.1]{clarke1990optimization} to get rid of the smooth part from our considerations. This reduces the computation to the establishment of the local surjection property of the map $\HI_{\operatorname{NeuMF}}:\R^{d\times h}\times\R^{d\times m}\times\R^{d\times h}\times\R^{d\times n}\times\R^{h}\rightarrow\R^{h m n}$ with
\begin{equation}\label{eq:mapping-neuralMF}
(\HI_{\operatorname{NeuMF}}(\bW,\bX,\bS,\bY,\bb))_{i j k}:=\bw_k^{\T}\bx_i+\bs_k^{\T}\by_j+b_k,\quad\text{for all $(i,j,k)\in[m]\times[n]\times[h]$}.
\end{equation}
However, as the following lemma shows, $\HI_{\operatorname{NeuMF}}$ can almost never possess the desired property.

\begin{lemma}
    % Suppose $\bbA:=\{(\be_k^{2H}\otimes\be_i^{U+I}+\be_{k+H}^{2H}\otimes\be_{j+U}^{U+I})\boxtimes\bI_d\}_{k,i,j=1}^{H,U,I}$ for some 
    % Let $H,m,n,d\in\N$. Then,
    If $\min\{m,n\}\ge 2$, then the map $\HI_{\operatorname{NeuMF}}$
    % $\HI(\{\bw_k\}_{k=1}^H,\{\bx_i\}_{i=1}^m,\allowbreak\{\bs_k\}_{k=1}^H,\{\by_j\}_{j=1}^n,\{b_k\}_{k=1}^H)=\operatorname{vec}(\bigvee_{k,i,j=1}^{H,m,n}(\bw_k^{\T}\bx_i+\bs_k^{\T}\by_j+b_k))$ from $(\R^{d})^{H+m+H+n}\times\R^H$ to $\R^{H m n}$
    (\ref{eq:mapping-neuralMF})
    is not locally surjective anywhere.
    % does not possess the local surjection property at any point. 
\end{lemma}

\begin{proof}
    % Let us parameterize $\bx:=(\bigvee_{k\in[H]}\bw_k)\vee(\bigvee_{k\in[H]}\bs_k)$ and $\by:=(\bigvee_{i\in[U]}\bx_i)\vee(\bigvee_{j\in[I]}\by_j)$. Then, it follows that $\HI(\bx,\by)=\bigvee_{k,i,j=1}^{H,U,I}(\bw_k^{\T}\bx_i+\bs_k^{\T}\by_j)$. 
    It is elementary to observe that for any $k\in[h]$ and $i,j\le\min\{m,n\}$ with $i\neq j$ (which do exist due to the assumption $\min\{m,n\}\ge 2$), we have
    $$
    (\bw_k^{\T}\bx_i+\bs_k^{\T}\by_i+b_k)+(\bw_k^{\T}\bx_j+\bs_k^{\T}\by_j+b_k)=(\bw_k^{\T}\bx_i+\bs_k^{\T}\by_j+b_k)+(\bw_k^{\T}\bx_j+\bs_k^{\T}\by_i+b_k).
    $$
    Hence, it follows that the range of $\HI_{\operatorname{NeuMF}}$ is contained within some proper subspace of $\R^{h m n}$.
    By a classical result in real analysis~\cite[Theorem~2.20(e)]{rudin1987rca}, we know the Lebesgue measure of any proper subspace must be zero.
    % , while the Lebesgue measure of any ball is strictly positive. 
    Thus, $\HI_{\operatorname{NeuMF}}$ can never be locally surjective anywhere.
    % , which completes the proof due to the same reason argued at the end of Example~\ref{ex:PITF}.
    % By~\cite[Theorem~2.20(e)]{rudin1987rca}, we know any proper subspace admits zero Lebesgue measure. Therefore, such a mapping can never possess the local surjection property at any point, as desired.
\end{proof}

% \begin{remark}
%     For the remaining case $\min\{m,n\}< 2$, we can indeed establish the local surjection property of $\HI$ under an overparameterized scenario. To see this, suppose w.l.o.g.\ that $n=1$. Then, $\HI(\bx,\by)$ reduces to $\bigvee_{k,i=1}^{H,m}(\bw_k^{\T}\bx_i+\bs_k^{\T}\by+b_k)$ (where we have used $\by$ to replace $\by_1$ for clarity). Then, by fixing $b_k=0$, looping $\bs_k$ and $\by$ over , and traversing $\bw_k$ and $\bp_i$ over some neighborhoods, and applying a similar argument as in the proof of Theorem~\ref{thm:surj-mf-anypoint}, we can get the desired claim.
% \end{remark}

Nevertheless, we think the most fundamental reason for the invalidation of the local surjection property of $\HI_{\operatorname{NeuMF}}$
% why our methodology does not apply for this particular example 
% the nonsmooth part of $L_{\operatorname{NeuMF}}$ 
is because it is not an MF in any conventional sense: The mapping $\HI_{\operatorname{NeuMF}}$ uses $\bw_k^{\T}\bx_i+\bs_k^{\T}\by_j+b_k$ to fit the $(i,j)$-th data, where $\bx_i$ and $\by_j$ are separated in two different terms, instead of something like $\bx_i^{\T}\by_j$ used in the vanilla MF.
% , which should of course hinder the generalization.

% The above argument can be easily generalized to deal with the so-called generalized tensor CP factorization model~\cite[Section~4.1]{chen2021neural}; we leave this to interested readers.
% We remark that the generalized MF also appears as a component in other modern neural networks such as~\cite{he2018nais}; see Equation~7 there.

% \subsubsection{Neural MF}

\subsubsection{Prospects}\label{sec:equiv-tensor}
Admittedly, the applications of the developed theories stop at the generalized MF. This is because for other related models, their studies more or less rely on the the local surjection property of $\HI_{\operatorname{CP}}$ (\ref{eq:mapping-CP})
% the mapping induced by the tensor CP factorization model
or its close relatives, which have not been fully understood. For such examples, the readers may refer to~\cite[Equation~1]{socher2013reasoning},~\cite[Equation~3]{blondel2016higher},~\cite[Equations~3-6]{he2017neuralfm},~\cite[Equation~5]{xiao2017attentional},~\cite[Equation~7]{he2018nais},~\cite[Equation~3]{xin2019relational},~\cite[Equation~5]{fan2019graph}, and~\cite[Section~4.1]{chen2021neural}. 
To be more specific, let us take the two-layer homogeneous neural FM~\cite{he2017neuralfm} as a concrete example, whose training objective function under the assumptions of Fact~\ref{fact:equiv} can be formulated as
$$
L_{\operatorname{NeuFM}}(\bv,\bP,\bH;\mathbb{T}):=\sum_{k=1}^n\ell_{k}\left(\sum_{s=1}^h v_{s}\cdot\sigma\Bigg(\sum_{j>i} x_{i k} x_{j k}\cdot \langle\bh_s,\bp_i,\bp_j\rangle\Bigg)\right),~\text{from $\R^{h}\times\R^{d\times d_0}\times \R^{d\times h}$ to $\R$}.
$$
% where $\mathbb{T}$ was defined in (\ref{eq:dataset-qualification}).
% where we have inherited all notations and definitions used before. 
% where $\mathbb{D}\subseteq\R^{d_0}$ is a finite set of training samples.
Under the data qualification in (\ref{eq:dataset-qualification}), it can be easily shown that as long as $\HI_{\operatorname{NeuFM}}(\bullet,\bullet;\ba)$ (\ref{eq:neuralFM})
% the following mapping $(\{\bh_k\}_{k=1}^H,\{\bp_i\}_{i=1}^{d_0})\mapsto\operatorname{vec}(\bigvee_{k\in[H],\ell>j}\langle\bh_k,\bp_j,\bp_{\ell}\rangle)$
possesses the local surjection property, we can easily compute $\partial_C L_{\operatorname{NeuFM}}(\bullet,\bullet,\bullet;\mathbb{T})$.
% the corresponding subdifferential chain rule of the linearized version of the above function will hold. 
% Unfortunately, we currently have no idea about under what conditions does such a mapping possess the desired property, but 
In view of the above discussions, 
% Therefore, 
% we firmly believe the bright prospects of
it should be meaningful to further understand the mappings in Section~\ref{sec:lsp-cp}. 
% can have a profound influence on the subdifferential theory of these models and even beyond.

\section{Concluding remarks}\label{sec:conclusion}
In this paper, we have studied the subdifferential chain rules of matrix factorization and factorization machine through the lens of local surjection property. 
Specifically, it has been shown for these problems that provided the latent dimension is larger than some multiple of the problem size (i.e., slightly overparameterized) and the loss function is locally Lipschitz, they are locally surjective (and therefore their subdifferential chain rules hold) everywhere.
In addition, we have examined the tightness of our conditions for the local surjection property through some interesting constructions and made some important observations from the perspective of optimization. Some tensor generalizations and neural extensions have also been discussed.

% Future works can be done in many aspects. 
This work also opens up several interesting (but can be rather challenging) future directions. Firstly, recall that in this paper we have considered a very general setting where the loss and activation functions are only required to be locally Lipschitz, which can be quite arbitrary. However, real-world learning models
% , especially in the field of deep learning, 
are usually composed of well-structured functions (e.g., the ReLU, which is piecewise-linear and convex).
% such as the well-known ReLU activation function $x\mapsto\max\{x,0\}$, 
% and they admit some special structures. 
Therefore, studying how to exploit such structures to further improve the conditions for subdifferential chain rules
% in Corollaries~\ref{cor:mf-chain-rule},~\ref{cor:fm-chain-rule}, such as developing some theory analogous to~\cite[Theorem~14]{tian2023testing} that was originally developed for the two-layer ReLU neural networks, 
should be interesting and meaningful. The second one stems from the following fact: Actually, 
% a slightly relaxed version of the local surjection property where
even if we replace the ball $\F(\bx)+r\bbB$ in Definition~\ref{def:lsp} by a dense subset of it, the resulting property would still suffice to guarantee the validity of the subdifferential chain rule~\cite[Theorem~2.3.10]{clarke1990optimization}.
% However, our negative constructions (cf.~Examples~\ref{ex:mf-general} and~\ref{ex:fm-general}) are not enough to 
% incapable of ruling out the possibility that the conditions in Corollaries~\ref{cor:mf-iff},~\ref{cor:fm-iff} are unnecessary (at least at the origin) for such a weak property since the two unhittable sets constructed therein both have zero Lebesgue measure, 
As a result, some further improvement may be possible along this direction. (We remark that our negative constructions (cf.~Examples~\ref{ex:mf-general} and~\ref{ex:fm-general}) are not strong enough to refute such a relaxed property.) In addition, continuing our study in Section~\ref{sec:lsp-cp} should be promising and important as well, as has already been explained in Section~\ref{sec:equiv-tensor}. 
Apart from the above, studying the subdifferential chain rule of symmetric matrix factorization, on which many interesting problems are based (see, e.g.,~\cite{li2020nonconvex,liu2024symmetric}), is also meaningful.
However, it is impossible to carry over our theories to this setting since even in the one-dimensional case the function $x\mapsto x^2$ can never be locally surjective at the origin due to its nonnegativity, and therefore new techniques are necessary. 
By the way, studying the local surjection property of other types of factorizations and products, such as the matrix tri-factorization~\cite{ding2005equivalence,ding2006orthogonal}, tensor-matrix mode-$n$ product~\cite[Section~2.5]{kolda2009tensor}, and tensor-tensor contracted product~\cite[Section~3.3]{bader2006algorithm}, should also be interesting. 
Finally, recall that in this paper we have only considered full factorizations
% the case where the dataset is full 
and directly reused their conditions to subsume all sparse cases (cf.~Remark~\ref{rmk:special-case}).
% the general theory with identical conditions, 
However, the sparsity
% for a reference in recommender systems), 
% and such an extreme sparsity 
might be particularly beneficial for our purposes as it relaxes the requirements for local surjection property. Besides, many datasets in real-world applications can indeed be extremely sparse; see, e.g.,~\cite[Table~4]{guan2024hybrid}. Therefore, studying if we can exploit the data sparsity to further improve our conditions is also promising and important.

\section*{Acknowledgments} 
The authors would like express their sincere gratitude to \href{https://www1.se.cuhk.edu.hk/~tianlai/}{Lai Tian} at the Chinese University of Hong Kong for very helpful discussions at various stages of the project, as well as \href{https://wangjinxin-terry.github.io/}{Jinxin Wang} for bringing the reference~\cite{levin2024effect} to their attention.

\bibliographystyle{ieeetr}
\bibliography{references.bib}{}

\end{document}